\documentclass{amsart}

\usepackage[utf8]{inputenc}
\usepackage{hyperref}
 \hypersetup{
     colorlinks=true,
     linkcolor=blue,
     citecolor=blue,
     urlcolor=blue,
 }

\usepackage[alphabetic, msc-links, backrefs]{amsrefs}
\usepackage{amssymb}
\usepackage{mathtools}
\usepackage{enumitem}
\usepackage{rsfso}
\usepackage{microtype}
\usepackage{upgreek}
\usepackage{scalerel,stackengine}

\usepackage{xcolor}
\usepackage{cancel}

\usepackage{doi}
\renewcommand{\PrintDOI}[1]{\doi{#1}}

\numberwithin{equation}{section}
\newtheorem{maintheorem}{Theorem}

\newtheorem{theorem}{Theorem}[section]
\newtheorem{lemma}[theorem]{Lemma}
\newtheorem{proposition}[theorem]{Proposition}
\newtheorem{corollary}[theorem]{Corollary}
\theoremstyle{definition}
\newtheorem{definition}[theorem]{Definition}
\theoremstyle{remark}
\newtheorem{remark}[theorem]{Remark}
\newtheorem{example}[theorem]{Example}
\newcommand{\cK}{{\mathcal K}}
\newcommand{\cM}{{\mathcal M}}

\newcommand{\e}{\varepsilon}
\newcommand{\g}{\gamma}
\newcommand{\de}{\delta}
\newcommand{\la}{\lambda}

\newcommand{\olt} {\overline t}

\newcommand{\olz} {\overline z}
\newcommand{\olx} {\overline x}

\newcommand{\La}{\Delta}
\newcommand{\Ga}{\Gamma}

\newcommand{\D}{\nabla}
\newcommand{\ra}{\rightarrow}
\newcommand{\pa}{\partial}
\newcommand{\supp}{\operatorname{supp}}
\newcommand{\dv}{\operatorname{div}}
\newcommand{\R}{\mathbb{R}}
\newcommand{\vp}{\varphi}

\newcommand{\al}{\alpha}
\newcommand{\be}{\beta}
\newcommand{\ka}{\kappa}
\newcommand{\si}{\sigma}
\newcommand{\sm}{\setminus}
\newcommand{\Om}{\Omega}

\newcommand{\N}{\widetilde{N}}

\newcommand{\wDu}{\widehat{\D u}}

\newcommand{\wu}{\widetilde{u}}
\newcommand{\wv}{\widetilde{v}}
\newcommand{\ww}{\widetilde{w}}
\newcommand{\we}{\widetilde{\varepsilon}}

\newcommand{\loc}{\mathrm{loc}}

\newcommand\reallywidehat[1]{\ThisStyle{%
  \setbox0=\hbox{$\SavedStyle#1$}%
  \stackengine{-1.0\ht0+.5pt}{$\SavedStyle#1$}{%
    \stretchto{\scaleto{\SavedStyle\mkern.15mu\char'136}{2.6\wd0}}{1.4\ht0}%
  }{O}{c}{F}{T}{S}%
}}

\newcommand{\mean}[1]{\langle{#1}\rangle}
\renewcommand{\Re}{\operatorname{Re}}

\newcommand{\fs}{\mathcal{F}}
\newcommand{\fso}{\mathcal{F}_0}
\newcommand{\fsz}{\mathcal{F}_{z_0}}

\def\Xint#1{\mathchoice
  {\XXint\displaystyle\textstyle{#1}}%
  {\XXint\textstyle\scriptstyle{#1}}%
  {\XXint\scriptstyle\scriptscriptstyle{#1}}%
  {\XXint\scriptscriptstyle\scriptscriptstyle{#1}}%
  \!\int}
\def\XXint#1#2#3{{\setbox0=\hbox{$#1{#2#3}{\int}$}
    \vcenter{\hbox{$#2#3$}}\kern-.5\wd0}}

\def\dashint{\Xint-}

\mathtoolsset{centercolon}

\allowdisplaybreaks[4]

\makeatletter
\@namedef{subjclassname@2020}{%
  \textup{2020} Mathematics Subject Classification}
\makeatother

\author{Seongmin Jeon}
\address{Department of Mathematics Education, Hanyang University, 222 Wangsimni-ro, Seongdong-gu, Seoul 04763, Republic of Korea}
\email{seongminjeon@hanyang.ac.kr}
\author{Arshak Petrosyan}
\address{Department of Mathematics, Purdue University, West Lafayette,
  IN 47907, USA}
\email{arshak@purdue.edu}
\thanks{S. Jeon was supported by the Academy of Finland grant 347550 and the research fund of Hanyang University(HY-202400000003278).}

\title[Free boundary for Almost minimizers of parabolic Signorini problem]{Free boundary regularity for almost minimizers of the parabolic Signorini problem}

\begin{document}
\begin{abstract}
In this paper, we study the regularity of the ``regular'' part of the free boundary for almost minimizers in the parabolic Signorini problem with zero thin obstacle. This work is a continuation of our earlier research on the regularity of almost minimizers. We first establish the Weiss-type monotonicity formula by comparing almost minimizers with parabolically homogeneous replacements and utilizing conformal self-similar coordinates. Subsequently, by deriving the Almgren-type frequency formula and applying the epiperimetric inequality, we obtain the optimal growth near regular free boundary points and achieve the regularity of the regular set.
\end{abstract}
\maketitle
\tableofcontents
\section{Introduction}

\subsection{Solutions of parabolic Signorini problem}
Let $\Omega$ be a domain in $\R^n$, $n\ge2$, and $\mathcal M$ be a smooth $(n-1)$-dimensional manifold that divides $\Omega$ into two parts: $\Omega\setminus\mathcal{M}=\Omega^+\cup \Omega^-$. For $T>0$, we set $\Omega_T:=\Omega\times(0,T]$, $\cM_T:=\cM\times(0,T]$ (\emph{the thin space}), and $(\pa\Omega)_T:=\pa\Omega\times(0,T]$. Let also $\vp:\cM_T\to\R$ (\emph{the thin obstacle}), $\vp_0:\Omega\times\{0\}\to\R$ (\emph{the initial value}), and $g:(\pa\Omega)_T\to\R$ (\emph{the boundary value}) be prescribed functions satisfying the compatibility conditions: $\vp_0\ge\vp$ on $\cM\times\{0\}$, $g\ge\vp$ on $(\cM\cap\pa\Omega)\times(0,T]$, and $g=\vp_0$ on $\pa\Omega\times\{0\}$.

We then say that a function $u\in W^{1,0}_2(\Omega_T)$ (see Subsection~\ref{subsec:not} for notations) is a solution of the \emph{parabolic thin obstacle} (or \emph{Signorini}) \emph{problem} in $\Omega_T$, if it satisfies the variational inequality 
\begin{align*}
    &\int_{\Omega_T}\D u\D(v-u)+\pa_tu(v-u)\ge0\quad\text{for any }v\in \cK,\\
    &u\in \mathcal{K},\ \pa_tu\in L^2(\Omega_T),\ u(\cdot,0)=\vp_0 \text{ on } \Omega,
\end{align*}
where $\mathcal K=\{v\in W^{1,0}_2(\Omega_T):v\ge\vp\text{ on }\cM_T,\ v=g\text{ on }(\pa\Omega)_T\}$. It is known that the solution $u$ satisfies 
\begin{align*}
    \Delta u-\pa_tu=0 & \text{ in }\Omega_T\setminus\cM_T,\\
    u\ge\vp,\ \pa_{\nu^+}u+\pa_{\nu^-}u\ge0,\ (u-\vp)(\pa_{\nu^+}u+\pa_{\nu^-}u)=0&\text{ on }\cM_T,\\
    u=g&\text{ on }(\pa\Omega)_T,\\
    u(\cdot,0)=\vp_0&\text{ on }\Omega\times\{0\},
\end{align*}
to be understood in a certain weak sense, where $\nu^\pm$ are the outer unit normal to $\Omega^\pm$ on $\cM$. 

In the study of the above problem, the main interests are
\begin{enumerate}[label=$\circ$]
    \item the regularity of the solution $u$,
    \item the regularity and structure of the \emph{free boundary}
    $$
    \Gamma(u)=\partial_{\cM_T}\{(x,t)\in \cM_T\,:\, u(x,t)>\vp(x,t)\}.
    $$
\end{enumerate}

A comprehensive examination of both regularity of the solution and the properties of the free boundary was conducted in \cite{DanGarPetTo17} by the second author, Danielli, Garofalo and To, under the condition that the thin manifold $\cM$ is flat (cf. refer to \cite{AryShi24} for the variable coefficients parabolic Signorini problem). Specifically, they established a generalized frequency formula, and employed it to achieve the optimal $H^{3/2,3/4}$ regularity of the solution and classified the free boundary points according to their frequency limits. \cite{DanGarPetTo17} dealt with two special types of subsets of the free boundary:  the \emph{regular set} and the \emph{singular set}.

The regular set is defined as the set of the free boundary points with minimal frequency $3/2$. Similarly to the elliptic Signorini problem \cites{AthCafSal08, CafSalSil08, PetShaUra12}, \cite{DanGarPetTo17} showed that there is a cone of spatial direction in which $u-\vp$ is monotone. Combining this with the fact that the blowups at regular points are time independent, they obtained the Lipschitz regularity of the regular set in the space variables. Moreover, by applying the parabolic boundary Harnack principles with thin Lipschitz complement, they proved that the regular set is given locally as a graph with $H^{\al,\al/2}$ regular gradient.

The singular set corresponds to the free boundary points with frequency $2m$, $m\in \mathbb{N}$, which have the characterization that the coincidence set $\{u=\vp\}$ has zero $H^n$-density in the thin manifold $\cM_T$. Following the approach in \cite{GarPet09} by the second author and Garofalo, \cite{DanGarPetTo17} established Weiss- and Monneau-type monotonicity formulas and applied the parabolic version of the Whitney's extension theorem to achieve the $C^1$ structure of the singular set.

\subsection{Almost minimizers}
In this paper we investigate the almost minimizers concerning the parabolic Signori problem described above. It serves as a continuation of \cite{JeoPet23}, where the authors previously explored the regularity of almost minimizers. For technical reasons, we consider two different notions of almost minimizers: ``unweighted'' almost minimizers and ``weighted'' almost minimizers. We first introduce unweighted almost minimizers, which correspond to those studied in \cite{JeoPet23}. 

We let $\vp$ be the thin obstacle on $Q'_1$. Given $r_0>0$, we say that $\eta:[0,r_0)\to[0,\infty)$ is a \emph{gauge function} or a modulus of continuity if $\eta$ is monotone nondecreasing and $\eta(0+)=0$. Here and henceforth we use notations from
Subsection~\ref{subsec:not}.

\begin{definition}[unweighted version]
Let $z_0=(x_0,t_0)\in Q_1$. We say that a function $u\in W^{1,1}_2(Q_1)$ satisfies the \emph{unweighted almost parabolic Signorini property} at $z_0$ if $u\ge\vp$ on $Q'_1$ and for any parabolic cylinder $Q_r(z_0)\Subset Q_1$ with $0<r<r_0$, we have
$$
\int_{Q_r(z_0)}(1-\eta(r))|\D u|^2+2\pa_tu(u-v)\le (1+\eta(r))\int_{Q_r(z_0)}|\D v|^2
$$
for any $v\in W^{1,0}_2(Q_r(z_0))$ with $v\ge \vp$ on $Q'_r(z_0)$ and $v-u\in L^2(t_0-r^2,t_0;W^{1,2}_0(B_r(x_0)))$.

We say that $u\in W^{1,1}_2(Q_1)$ is an \emph{unweighted almost minimizer for the parabolic Signorini problem} in $Q_1$ if $u\ge\vp$ on $Q_1'$ and $u$ satisfies the almost parabolic Signorini property at every $z_0\in Q_1$.
\end{definition}

\medskip
Next, we define the weighted version of almost minimizers. To this aim, we observe that if $u$ is a solution of the parabolic Signorini problem in the strip $S_1$, then
\begin{align}\label{eq:par-Sig}
    \int_{S_1}\left[(-t)|\D u|^2+(-x\cdot\D u-2t\partial_tu)(u-w)\right]G\,dxdt\le \int_{S_1}(-t)|\D w|G\,dxdt
\end{align}
for every proper competitor $w$. This motivates the following definition of weighted almost minimizers.

\begin{definition}[weighted version]
Let $z_0=(x_0,t_0)\in Q'_1$. We say that a function $u\in \fsz$ satisfies the \emph{weighted almost parabolic Signorini property} at $z_0$ if $u\ge\vp$ on $S'_1$ and \begin{align}\label{eq:par-alm-min-def}\begin{split}
    &\begin{multlined}\int_{S_r(t_0)\sm S_\rho(t_0)}\big[ (1-\eta(r))(t_0-t)|\D u|^2\\
    +(-(x-x_0)\cdot \D u-2(t-t_0)\pa_tu)(u-w)\big]G_{z_0}\,dxdt\end{multlined}\\
    &\qquad\begin{multlined}\le (1+\eta(r))\int_{S_r(t_0)\sm S_\rho(t_0)} (t_0-t)|\D w|^2G_{z_0}\,dxdt\\
    +\eta(r)\int_{S_r(t_0)\sm S_\rho(t_0)}(u-w)^2G_{z_0}\,dxdt+\|u\|_{\fsz}^2e^{-\frac1r}.\end{multlined}\end{split}\end{align}
for any $0\le \rho<r<r_0$ with $-1<t_0-r^2$, and $w\in L^2(t_0-r^2,t_0-\rho^2;W^{1,2}(\R^n,G_{z_0}))$ with $w\ge \vp$ on $S'_r(t_0)\sm S'_\rho(t_0)$ and $u-w\in L^2(t_0-r^2,t_0-\rho^2;W^{1,2}_0(\R^n,G_{z_0}))$.

We say that a function $u\in \fs$ is a \emph{weighted almost minimizer for the parabolic Signori problem} on $Q'_1$ if $u\ge\vp$ on $S'_1$ and $u$ satisfies the weighted almost parabolic Signorini property at every $z_0\in Q'_1$.
\end{definition}

The readers might be intrigued by the presence of the exponential term $\|u\|_{\fsz}^2e^{-\frac1r}$ in \eqref{eq:par-alm-min-def}. We incorporated this term since we discovered that solutions of some perturbed parabolic Signorini problems exhibit characteristics of almost minimizers, with the inclusion of the exponential error, see Appendix~\ref{sec:appen-ex}.

\begin{definition}
We say that a function $u\in \fs$ is an \emph{almost minimizer for the parabolic Signori problem} in $Q_1$ if it is both an unweighted almost minimizer in $Q_1$ and a weighted almost minimizer on $Q'_1$.
\end{definition}

The notion of a weighted almost minimizer is crucial for establishing monotonicity formulas, which are significant ingredients in our analysis of the free boundary. We will verify in Appendix~\ref{sec:appen-ex} that solutions of some perturbed parabolic Signorini problems, multiplied by a standard cutoff function, satisfy the weighted almost parabolic Signorini property.

For background information and relevant literature concerning almost minimizers, we refer to \cite{JeoPet23} and references therein.

The time-independent almost minimizers for the Signorini problem were comprehensively treated by the authors in \cite{JeoPet21}. This paper extends specific results from the elliptic to the parabolic setting by employing similar energy methods. However, the parabolic case presents significant new challenges compared to the elliptic setting, primarily because we have to work with energy functionals involving singular weights.

\subsection{Main results}
This paper focuses on the local regularity results for free boundaries. Thus we assume that the domain $\Omega_T\subset \R^n\times\R$ is the parabolic cylinder $Q_1$. Given the technical nature of the problem, we specifically examine the scenario where the thin space $\cM_T$ is $Q'_1$ (flat thin space), the thin obstacle $\vp$ is identically zero (zero thin obstacle), and the gauge function $\eta(r)=r^\al$ for some $0<\al<1$ with $r_0=1$.

\medskip

Our first central result of this paper concerns the Weiss-type monotonicity formula.

\begin{maintheorem}
    \label{mainthm:par-Weiss}
    Fix $\ka_0>2$, $0<\de<2$ and $0<\e\le\al<1$. For $z_0\in\Gamma(u)\cap  Q_{1/2}'$, let $u\in \fsz$ satisfy the weighted almost parabolic Signorini property at $z_0$. For $0<\ka<\ka_0$, we set
    \begin{multline*}
W_{\ka,\al,\e,\de}(r,u,z_0)\\
\begin{multlined}:=\frac{e^{ar^\al}}{r^{2\ka+2}}\Bigg(\int_{S_r(t_0)}\left(2(t_0-t)|\D u|^2-\ka(1-br^\e)u^2\right)G_{z_0}\,dxdt+\|u\|_{\fsz}^2e^{-\frac1{r}}r^{-\de}\Bigg),\end{multlined}
\end{multline*}
where $a=a(\ka,\al)>0$ and $b=b(\ka,\e)>0$ are as in Theorem~\ref{thm:par-weiss}. Then $W_{\ka,\al,\e,\de}(r,u,z_0)$ is nondecreasing in $r$ for $0<r<r_0=r_0(\ka_0,\e)$.
\end{maintheorem}
Since almost minimizers do not satisfy a partial differential equations, we prove Theorem~\ref{mainthm:par-Weiss} by comparing them with appropriate homogeneous replacements, as done in the elliptic counterpart \cite{JeoPet21}. However, in our parabolic case, we encounter new technical challenges, making the proof significantly more complicated. This is mainly due to the formulation of the Weiss-type energy, which is defined in the unbounded strip and involves the singular weight. It turns out that we have to employ conformal self-similar coordinates.

By making use of the above one-parameter family of Weiss-type monotonicity formulas, we derive the Almgren-type frequency formula. For caloric functions, the monotonicity of the following frequency was established in \cite{Poo96}:
$$
r\longmapsto N(r,u,z_0)=\frac{r^2\int_{\R^n\times\{t_0-r^2\}}|\D u|^2G_{z_0}dx}{\int_{\R^n\times\{t_0-r^2\}}u^2G_{z_0}dx}.
$$
Recently, its averaged version was considered in \cite{DanGarPetTo17} for the study of the parabolic Signorini problem. Regarding almost minimizers, we show that a modification of those quantities is monotone. To describe it, we denote
$$
 N_\de(r,u,z_0):=\frac{\int_{S_r(t_0)}2(t_0-t)|\D u|^2G_{z_0}+\|u\|_{\fsz}^2e^{-\frac1{r}}r^{-\de}}{\int_{S_r(t_0)}u^2G_{z_0}}.
$$
\begin{maintheorem}[Almgren-type monotonicity formula]\label{mainthm:par-almgren}
Let $\ka_0,\delta,\e,\al,z_0,u,$ and $b$ be as in Theorem~\ref{mainthm:par-Weiss}. Then $\widehat{N}_{\ka_0,\e,\de}(r,u,z_0):=\min\{\frac1{1-br^\e}N_\de(r,u,z_0),\ka_0\}$ is nondecreasing in $0<r<r_0=r_0(\ka_0,\e)$. Moreover, if $u$ is even-symmetric in $x_n$-variable, then we have either
$$
\widehat{N}_{\ka_0,\de}(0+,u,z_0)\footnote{From the monotonicity of $\widehat{N}_{\ka_0,\e,\delta}$ and $\lim_{r\to0}(1-br^\e)=1$, we see that the limit $\widehat{N}_{\ka_0,\de}(0+,u,z_0)=\lim_{r\to 0}\widehat{N}_{\ka_0,\e,\de}(r,u,z_0)$ exists and its value is independent of $\e$.}=3/2\quad\text{or}\quad \widehat{N}_{\ka_0,\de}(0+,u,z_0)\ge2.
$$
\end{maintheorem}
In contrast to the Almgren frequency utilized for solutions to the parabolic Signorini problem \cite{DanGarPetTo17}, the Almgren-type frequencies we work with for almost minimizers include the extra exponential term $\|u\|_{\fsz}^2e^{-\frac1r}r^{-\de}$. Yet, we will show that this term is unsubstantial (see Lemma~\ref{lem:par-Almgren}) and derive the same minimal frequency and frequency gap as presented in \cite{DanGarPetTo17} (see Lemma~\ref{lem:par-min-freq}).

Next, we consider the subset of the free boundary
\begin{align*}
    \mathcal{R}(u)=\{z_0\in \Gamma(u)\cap Q'_{1/2}\,:\, \widehat{N}_{\ka_0,\delta}(0+,u,z_0)=3/2\text{ for some $\ka_0>2$, $0<\delta<2$}\},
\end{align*}
the set of all free boundary points with the minimal frequency $3/2$, known as the \emph{regular set}.

\begin{maintheorem}[Optimal growth near regular free boundary]\label{mainthm:par-opt-growth-est}
Fix $\ka_0>2$. Suppose that an even-symmetric funtion $u\in \fsz$ satisfies the almost parabolic Signorini property at $z_0\in\mathcal{R}(u)$. Then, 
\begin{align*}
    \int_{S_r(t_0)}u^2G_{z_0}\,dxdt &\le C(\ka_0,n,\al)\|u\|_{\fsz}^2r^{5},
\end{align*}
for $0<r<r_0=r_0(\ka_0,n,\al)$.
\end{maintheorem}

In the elliptic counterpart \cite{JeoPet21}, an analogous result was derived using the epiperimetric inequality. Regarding the parabolic Signorini problem, Shi \cite{Shi20} obtained a similar result by introducing the parabolic epiperimetric inequality. In our case, we adopt similar approaches. It is worth noting that while the application of these inequalities is rather immediate or standard in \cites{Shi20, JeoPet21}, it is considerably more complicated for the parabolic almost minimizers (see Lemmas \ref{lem:par-weiss-bound}-\ref{lem:par-opt-growth-est}).

Finally, the main result concerning the regularity of the regular set is as follows.

\begin{maintheorem}[Regularity of the regular set]\label{mainthm:par-Clg-reg-set}
Let $u\in \fs$ be a symmetric almost minimizer for the parabolic Signorini problem in $Q_1$. Then $\mathcal{R}(u)$ can be represented locally as an $(n-2)$-dimensional graph of a function, which has Hölder continuous spatial derivatives.
\end{maintheorem}

\subsubsection{Proofs of Theorems~\ref{mainthm:par-Weiss}--\ref{mainthm:par-Clg-reg-set}} Although we
do not provide formal proofs of Theorems~\ref{mainthm:par-Weiss}--\ref{mainthm:par-Clg-reg-set} in
the main body of the paper, they can be deduced from the combination of
results there. To be more precise,
\begin{enumerate}[label=$\circ$,leftmargin=*,labelindent=\parindent]
\item Theorem~\ref{mainthm:par-Weiss} is contained in Theorem~\ref{thm:par-weiss}.
\item Theorem~\ref{mainthm:par-almgren} follows by combining
  Theorem~\ref{thm:par-almgren} and Lemma~\ref{lem:par-min-freq}.
\item The statement of Theorem~\ref{mainthm:par-opt-growth-est} is contained in that of Lemma~\ref{lem:par-opt-growth-est}.
\item The statement of Theorem~\ref{mainthm:par-Clg-reg-set} is contained in that of
  Theorem~\ref{thm:par-Clg-reg-set}.
\end{enumerate}


\section{Notation and Preliminaries}

\subsection{Notation}\label{subsec:not}
We use the following notations throughout the paper.

For a function $u$, a set $\Om\subset \R^{n+1}$, a constant $\e\in(0,1)$, and a point $z_0=(x_0,t_0)$, we denote
\begin{align*}
    Q_r(z_0)&=B_r(x_0)\times(t_0-r^2,t_0]\\
    Q_{r,\rho}^\e(z_0)&=B_{r^\e}(x_0)\times (t_0-r^2,t_0-\rho^2]\\
    \pa_pQ_r(z_0)&=\left(\pa B_r(x_0)\times[t_0-r^2,t_0]\right)\cup\left(B_r(x_0)\times\{t_0-r^2\}\right)\,:\, \text{parabolic boundary}\\
    S_\rho(t_0)&=\R^n\times(t_0-\rho^2,t_0]\\
    \Om'&=\Om\cap \{x_n=0\}\\
    u_\Om&=\dashint_\Om u\\
    u_{z_0,r}&=u_{Q_r(z_0)}=\dashint_{Q_r(z_0)}u\\
    \|z_0\|&=\left(|x_0|^2+|t_0|\right)^{1/2}\,:\, \text{parabolic norm}\\
    \Ga(u)&=\pa_{Q'_1}\{(x',t)\in Q_1'\,:\, u(x',0,t)=0\}\,:\,\text{free boundary}
\end{align*}

Given $l=k+\gamma$ with $k\in \mathbb{N}\cup\{0\}$ and $0<\gamma\le1$, we use standard notations for parabolic H\"older spaces of functions $H^{l,l/2}$. For $1\le q\le \infty$, we denote $W^{1,0}_q$ and $W^{1,1}_q$ by standard parabolic Sobolev spaces of functions.
We refer to \cites{DanGarPetTo17, JeoPet23} for detailed definition.

We denote the backward heat kernel by
$$
G(x,t)=\begin{cases}
        (-4\pi t)^{-n/2}e^{\frac{|x|^2}{4t}},&t<0\\
        0,&t\ge0,
    \end{cases} 
$$
and write its translations 
$$
G_{z_0}=G(\cdot-x_0,\cdot-t_0).
$$
Given $z_0=(x_0,t_0)\in Q_1'$ and $0<r<1$, we let
\begin{align*}
&\|u\|_{W^{1,0}_2(S_r(t_0),G_{z_0})}:=\left[\int_{S_r(t_0)}\left(u^2+(t_0-t)|\D u|^2\right)G_{z_0}\,dxdt\right]^{1/2},\\
&\|u\|_{W^{1,1}_2(S_r(t_0),G_{z_0})}:=\left[\int_{S_r(t_0)}\left(u^2+(t_0-t)\left(|\D u|^2+(\partial_tu)^2\right)\right)G_{z_0}\,dxdt\right]^{1/2}.
\end{align*}
We say that $u\in \fsz$ if $u\in W^{1,1}_2(\R^n\times(-1,t_0),G_{z_0})\cap W^{1,1}_2(B_1\times(-1,t_0))\cap L^\infty(\R^n\times(-1,t_0))$. We define the associated norm by
$$
\|u\|_{\fsz}:=\|u\|_{W^{1,1}_2(\R^n\times(-1,t_0),G_{z_0})}+\|u\|_{W^{1,1}_2(B_1\times(-1,t_0))}+\|u\|_{L^\infty(\R^n\times(-1,t_0))}.
$$
In addition, we say that $u\in \fs$ if $u\in \fsz$ for every $z_0\in Q'_1$ and
$$
\|u\|_{\fs}:=\sup_{z_0\in Q'_1}\|u\|_{W^{1,1}_2(\R^n\times(-1,t_0),G_{z_0})}+\|u\|_{W^{1,1}_2(Q_1)}+\|u\|_{L^\infty(S_1)}<\infty.
$$


\subsection{Preliminaries}

The following regularity result for unweighted almost minimizers was proved in \cite{JeoPet23}.

\begin{theorem}\label{thm:par-alm-min-grad-holder}
Let $u$ be an unweighted almost minimizer for the parabolic Signorini problem in $Q_1$. Then 
\begin{enumerate}
    \item $u\in H^{\sigma,\sigma/2}(Q_1)$ for every $0<\sigma<1$;
    \item $\D u\in H^{\be,\be/2}(Q_1^\pm\cup Q'_1)$ for some $\be=\be(n,\al)>0$.
\end{enumerate}
\end{theorem}

Moreover, the authors showed in \cite{JeoPet23} that if $u$ is an almost caloric function, then a stronger result than $(2)$ in Theorem~\ref{thm:par-alm-min-grad-holder} holds:
$$
\D u\in H^{\al/2,\al/4}(Q_1).
$$
Here, an almost caloric function essentially is an unweighted almost minimizer without the obstacle condition; we refer to \cite{JeoPet23}*{Definitions~2.1-2.2} for its precise definition and \cite{JeoPet23}*{Theorem~2.8} for its regularity result.

By using Theorem~\ref{thm:par-alm-min-grad-holder} and the above Hölder continuity of spatial gradients of almost caloric functions across the thin space $Q'_1$, we can follow the argument in \cite{JeoPet21}*{Lemma~4.7} to derive the following complementarity condition.

\begin{lemma}[Complementarity condition]\label{lem:par-compl-cond}
Let $u$ be an unweighted almost minimizer for the parabolic Signorini problem in $Q_1$, even in $x_n$-variable. Then $u$ satisfies the following complementarity condition $$
u\pa_{x_n}^+u=0\quad\text{on }Q'_1.
$$
In addition, we define
$$
\wDu(x',x_n,t):=\begin{cases}
    \D u(x',x_n,t),&x_n\ge0,\\
    \D u(x',-x_n,t),&x_n<0,
\end{cases}
$$
the even extension of $\D u$ from $Q_1^+$ to $Q_1$. If $z_0\in \Ga(u)$, then $$
u(z_0)=0\quad\text{and}\quad |\wDu(z_0)|=0.
$$
\end{lemma}


\section{Weiss- and Almgren-type monotonicity formulas}
The purpose of this section is to establish monotonicity formulas of Weiss- and Almgren-type. They will play a crucial role in the analysis of the free boundary.

We first prove the Weiss-type monotonicity formula, which represents one of the most technical aspects of this paper. In its elliptic counterpart \cite{JeoPet21}, the authors derived the formula by comparing almost minimizers and homogeneous replacements, inspired by the approach in \cite{Wei99}. In the current parabolic case, we compare almost minimizers and parabolically homogeneous replacements, and utilize conformal self-similar coordinates. For its proof, we need the following auxiliary results.

\begin{lemma}\label{lem:par-weiss-lem}
Fix $\ka_0>2$ and $0<\e\le\al<1$. For $0<\ka<\ka_0$ and $0\le\rho<r$, let 
$$
\Phi(r):=\Phi_{\rho,\ka,\al}(r)=\frac{e^{ar^\al}}{r^{2\ka+2}-\rho^{2\ka+2}},\quad \Psi(r):=\Psi_{\rho,\ka,\ka_0,\al,\e}(r)=\frac{(1-br^\e)e^{ar^\al}}{r^{2\ka+2}-\rho^{2\ka+2}}
$$
with
$$   
a=\frac{8(\ka+1)}{\al},\quad b=\frac{128(\ka_0+1)}{\e}.
$$
Then, there is a small constant $r_0=r_0(\ka_0,\e)=\frac{r_0(\e)}{\ka_0^{2/\e}}>0$ such that for $0\le\rho<r<r_0$ with $\rho/r\le 1/\sqrt{2}$,
\begin{align}
    \label{eq:par-weiss-lem-1}
    &\Phi'(r)\le 0,\\
    \label{eq:par-weiss-lem-2}
    & \frac{\Phi'(r)}{1-r^\al}-\Psi'(r)\ge -\frac{\left(2\ka+2-\e/4\right)b\Phi(r)r^{2\ka+1+\e}}{r^{2\ka+2}-\rho^{2\ka+2}},\\
    \label{eq:par-weiss-lem-3}
    &\frac{1+r^\al}{1-r^\al}\Phi'(r)+\frac{2(\ka+1)r^{2\ka+1}}{r^{2\ka+2}-\rho^{2\ka+2}}\Phi(r)\ge 0,\\
    \label{eq:par-weiss-lem-4}
    &-\frac{\Phi'(r)}{1-r^\al}-\frac{2(\ka+1)r^{2\ka+1}}{r^{2\ka+2}-\rho^{2\ka+2}}\Psi(r)\ge \frac{\left(2\ka+2-\e/8\right)b\Phi(r)r^{2\ka+1+\e}}{r^{2\ka+2}-\rho^{2\ka+2}}.
\end{align}
\end{lemma}

\begin{proof}
We first prove \eqref{eq:par-weiss-lem-1}. By using $0<\e<\al$, we simply compute
\begin{align*}
    \Phi'(r)&=\left(a\al r^\al-\frac{(2\ka+2)r^{2\ka+2}}{r^{2\ka+2}-\rho^{2\ka+2}}\right)\frac{\Phi(r)}r\\
    &\le ((8\ka+8)r^\al-(2\ka+2))\frac{\Phi(r)}r\le 0,\quad r<r_0(\e).
\end{align*}
For \eqref{eq:par-weiss-lem-2}, we note $r^{2\ka+2}-\rho^{2\ka+2}\ge \left(1-\left(\rho/r\right)^2\right)r^{2\ka+2}\ge \frac12r^{2\ka+2}$ and get
\begin{align*}
\left(r^{2\ka+2}-\rho^{2\ka+2}\right)\Phi'(r)&=\left(\left(r^{2\ka+2}-\rho^{2\ka+2}\right)a\al r^\al-(2\ka+2)r^{2\ka+2}\right)\frac{\Phi(r)}r\\
&\ge \left(1/2a\al r^\al-(2\ka+2)\right)\Phi(r)r^{2\ka+1}.
\end{align*}
Moreover, using $r^{2\ka+2}-\rho^{2\ka+2}\ge \frac12r^{2\ka+2}$ again along with $b\e r^\e\ge a\al r^\al\ge (1-br^\e)a\al r^\al$, we find
\begin{align*}
    &\left(r^{2\ka+2}-\rho^{2\ka+2}\right)\Psi'(r)\\
    &=\left(r^{2\ka+2}-\rho^{2\ka+2}\right)\left(-b\e r^{\e-1}\Phi(r)+(1-br^\e)\Phi'(r)\right)\\
    &=\left(\left(r^{2\ka+2}-\rho^{2\ka+2}\right)(-b\e r^\e+(1-br^\e)a\al r^\al)-(1-br^\e)(2\ka+2)r^{2\ka+2}\right)\frac{\Phi(r)}r\\
    &\le\left(1/2(-b\e r^\e+(1-br^\e)a\al r^\al)-(1-br^\e)(2\ka+2)\right)\Phi(r)r^{2\ka+1}. 
\end{align*}
Thus, we have \begin{align*}
    &(r^{2\ka+2}-\rho^{2\ka+2})\left(\frac{\Phi'(r)}{1-r^\al}-\Psi'(r)\right)\\
    &\begin{multlined}[t]\ge\bigg(1/2a\al r^\al-(2\ka+2)+\frac{1/2a\al}{1-r^\al}r^{2\al}-\frac{2\ka+2}{1-r^\al}r^\al\\
    +1/2b\e r^\e-1/2(1-br^\e)a\al r^\al+(2\ka+2)-(2\ka+2)br^\e\bigg)\Phi(r)r^{2\ka+1}\end{multlined}\\
    &\ge -\left((2\ka+2-\e/2)br^\e+\frac{2\ka+2}{1-r^\al}r^\al\right)\Phi(r)r^{2\ka+1}\\
    &\ge -\left(2\ka+2-\e/4\right)br^\e\Phi(r)r^{2\ka+1},
\end{align*}
where the last inequality follows from $\frac{2\ka+2}{1-r^\al}\le16(\ka+1)\le b\e/8$ and $r^\al\le r^\e$.
Regarding \eqref{eq:par-weiss-lem-3}, we use $r^{2\ka+2}-\rho^{2\ka+2}\ge\frac12r^{2\ka+2}$ once again to obtain
\begin{align*}
    &\frac{1+r^\al}{1-r^\al}\Phi'(r)+\frac{2(\ka+1)r^{2\ka+1}}{r^{2\ka+2}-\rho^{2\ka+2}}\Phi(r)\\
    &=\left(\frac{1+r^\al}{1-r^\al}\left(a\al r^\al-\frac{2(\ka+1)r^{2\ka+2}}{r^{2\ka+2}-\rho^{2\ka+2}}\right)+\frac{2(\ka+1)r^{2\ka+2}}{r^{2\ka+2}-\rho^{2\ka+2}}\right)\frac{\Phi(r)}r\\
    &=\left((1+r^\al)a\al r^\al-\frac{4(\ka+1)r^{2\ka+2+\al}}{r^{2\ka+2}-\rho^{2\ka+2}}\right)\frac{\Phi(r)}{(1-r^\al)r}\\
    &\ge \left((1+r^\al)a\al-8(\ka+1)\right)\frac{\Phi(r)}{(1-r^\al)r^{1-\al}}\ge 0.
\end{align*}
Finally, we prove \eqref{eq:par-weiss-lem-4}.
\begin{align*}
    &(r^{2\ka+2}-\rho^{2\ka+2})\left(-\frac{\Phi'(r)}{1-r^\al}-\frac{2(\ka+1)r^{2\ka+1}}{r^{2\ka+2}-\rho^{2\ka+2}}\Psi(r)\right)\\
    &=\left(-\frac{a\al}{1-r^\al}(r^{2\ka+2}-\rho^{2\ka+2})r^\al+\frac{2\ka+2}{1-r^\al}r^{2\ka+2}-(1-br^\e)(2\ka+2)r^{2\ka+2}\right)\frac{\Phi(r)}r\\
    &\ge\left(-\frac{a\al}{1-r^\al}r^\al+(2\ka+2)br^\e\right)\Phi(r)r^{2\ka+1}\ge \left(2\ka+2-\e/8\right)br^\e\Phi(r)r^{2\ka+1},
\end{align*}
where the last step follows from $\frac{a\al}{1-r^\al}\le 16(\ka+1)\le b\e/8$ and $r^\al\le r^\e$.
\end{proof}

As previously mentioned, we will make use of conformal self-similar coordinates. Given constants $0<r<1$ and $\ka>0$ and a function $u$ defined in $S_r$, we define
\begin{align}
    \label{eq:chan-coor}
    \wu(y,\tau)=\wu_\ka(y,\tau):=e^{\ka\tau/2}u\left(2e^{-\tau/2}y,-e^{-\tau}\right), \quad (y,\tau)\in\R^n\times(-2\ln r,\infty).
\end{align}
In addition, we let
\begin{align}
    \label{eq:hom-rep}
    w(x,t):=\left(\frac{\sqrt{-t}}{r}\right)^\ka u\left(\frac{r}{\sqrt{-t}}x,-r^2\right),\quad (x,t)\in S_r
\end{align}
be the parabolically $\ka$-homogeneous replacement of $u$ in $S_r$. From its construction, it is easily seen that $w$ satisfies the homogeneity
\begin{align}
    \label{eq:hom-rep-hom}
    \ka w-x\cdot\D w-2t\partial_tw=0.
\end{align}
Then, $\ww(y,\tau):=e^{\ka\tau/2}w\left(2e^{-\tau/2}y,-e^{-\tau}\right)$ satisfies
\begin{align}
    \label{eq:hom-rep-time-deriv}
    \pa_\tau\ww(y,\tau)=0\quad\text{for } (y,\tau)\in\R^n\times(-2\ln r,\infty),
\end{align}
which implies that $\ww(y)=\ww(y,\tau)$ is independent of $\tau$-variable. This, along with the fact that $w(x,-r^2)=u(x,-r^2)$ for $x\in\R^n$, yields
\begin{align}
    \label{eq:hom-rep-iden}
    \ww(y)=\wu(y,-2\ln r),\quad y\in\R^n.
\end{align}

\begin{lemma}\label{lem:par-(u-w)-transf}
Let $u\in \fso$. Then, for $\ka>0$ and $0\le\rho<r<1$, 
\begin{align}
    \label{eq:par-(u-w)-transf-1}\begin{split}
    &\int_{S_r\sm S_\rho}(\ka u-x\cdot\D u-2t\pa_tu)(u-w)G\\
    &=\frac{\rho^{2\ka+2}}{\pi^{n/2}}\int_{\R^n}(\wu(y,-2\ln\rho)-\wu(y,-2\ln r))^2e^{-|y|^2}\,dy+(\ka+1)\int_{S_r\sm S_\rho}(u-w)^2G,
\end{split}\end{align}
where $\wu$ and $w$ are as in \eqref{eq:chan-coor} and \eqref{eq:hom-rep}, respectively. In particular, \begin{align}
    \label{eq:par-(u-w)-transf-2}
    \int_{S_r}(\ka u-x\cdot\D u-2t\pa_tu)(u-w)G=(\ka+1)\int_{S_r}(u-w)^2G.
\end{align}

\end{lemma}

\begin{proof}
By using \eqref{eq:hom-rep-hom} and \eqref{eq:hom-rep-iden}, we obtain \eqref{eq:par-(u-w)-transf-1}:
\begin{align*}
    &\int_{S_r\sm S_\rho}(\ka u-x\cdot \D u-2t\pa_tu)(u-w)G\,dxdt\\
    &=\int_{S_r\sm S_\rho}(\ka(u-w)-x\cdot\D(u-w)-2t\pa_t(u-w))(u-w)G\,dxdt\\
    &=\frac12\int_{S_r\sm S_\rho}(2\ka(u-w)^2-x\cdot\D((u-w)^2)-2t\pa_t((u-w)^2))G\,dxdt\\
    &=\frac1{\pi^{n/2}}\int_{\R^n\times(-2\ln r,-2\ln\rho)}\pa_\tau((\wu-\ww)^2)e^{-|y|^2}e^{-(\ka+1)\tau}\,dyd\tau\\
    &=\frac1{\pi^{n/2}}\int_{\R^n\times\{-2\ln\rho\}}(\wu-\ww)^2e^{-|y|^2}\rho^{2\ka+2}\,dy\\
    &\qquad+\frac{\ka+1}{\pi^{n/2}}\int_{\R^n\times(-2\ln r,-2\ln\rho)}(\wu-\ww)^2 e^{-|y|^2}e^{-(\ka+1)\tau}\,dyd\tau\\
    &=\frac{\rho^{2\ka+2}}{\pi^{n/2}}\int_{\R^n}(\wu(y,-2\ln\rho)-\wu(y,-2\ln r))^2e^{-|y|^2}\,dy+(\ka+1)\int_{S_r\sm S_\rho}(u-w)^2G.
\end{align*}
Moreover, \eqref{eq:par-(u-w)-transf-2} follows from \eqref{eq:par-(u-w)-transf-1} by taking $\rho\to0$ with the observation 
\begin{equation*}
\lim_{\rho\to 0}\rho^{2\ka+2}(\wu(y,-2\ln\rho)-\wu(y,-2\ln r))^2\le \lim_{\rho\to 0}\rho^{2\ka+2}((\rho^{-\ka}+r^{-\ka})\|u\|_{L^\infty(S_r)})^2=0.
\end{equation*}
\end{proof}

We now prove the Weiss-type monotonicity formula with the help of Lemmas~\ref{lem:par-weiss-lem} and \ref{lem:par-(u-w)-transf}. We note that for any $\ka>0$, the weighted almost parabolic Signorini property \eqref{eq:par-alm-min-def} is equivalent to
\begin{align}\label{eq:par-alm-min-def-kappa}\begin{split}
    &\begin{multlined}\int_{S_r(t_0)\sm S_\rho(t_0)}\big[ 2(1-\eta(r))(t_0-t)|\D u|^2-\ka u^2\\
    +2(\ka u-(x-x_0)\cdot \D u-2(t-t_0)\pa_tu)(u-w)\big]G_{z_0}\,dxdt\end{multlined}\\
    &\qquad\begin{multlined}\le \int_{S_r(t_0)\sm S_\rho(t_0)}\big[ 2(1+\eta(r))(t_0-t)|\D w|^2-\ka w^2\\+\left(\ka+2\eta(r)\right)(u-w)^2\big]G_{z_0}\,dxdt+2\|u\|_{\fsz}^2e^{-\frac1r}.\end{multlined}
\end{split}\end{align}

\begin{theorem}[Weiss-type monotonicity formula]\label{thm:par-weiss} Fix $\ka_0>2$, $0<\de<2$ and $0<\e\le\al<1$. Suppose that for $z_0=(x_0,t_0)\in Q'_{1/2}$, $u\in \fsz$ satisfies the weighted almost parabolic Signorini property at $z_0$. For $0<\ka<\ka_0$, set 
\begin{multline*}
W_{\ka,\al,\e,\de,\rho}(r,u,z_0)\\
\begin{multlined}:=\frac{e^{ar^\al}}{r^{2\ka+2}-\rho^{2\ka+2}}\Bigg(\int_{S_r(t_0)\sm S_\rho(t_0)}\left(2(t_0-t)|\D u|^2-\ka(1-br^\e)u^2\right)G_{z_0}\,dxdt\\+\|u\|_{\fsz}^2e^{-\frac1{r}}r^{-\de}\Bigg),\end{multlined}
\end{multline*}
where constants $a,b$ are as in Lemma~\ref{lem:par-weiss-lem}
$$ 
a=\frac{8(\ka+1)}{\al},\quad b=\frac{128(\ka_0+1)}{\e}.
$$
\begin{enumerate}[label=\textup{(\roman*)},leftmargin=*,widest=ii,align=left]
    \item For $0<\rho<r<r_0=r_0(\ka_0,\e)=\frac{r_0(\e)}{\ka_0^{2/\e}}$ with $\rho/r\le 1/\sqrt{2}$, \begin{multline}
        \label{eq:par-weiss-deriv-1}
        \frac{d}{dr}W_{\ka,\al,\e,\de,\rho}(r,u,z_0)\\
        \ge \frac{(4\ka+2)r^{2\ka+1}\rho^{2\ka+2}}{\pi^{n/2}(r^{2\ka+2}-\rho^{2\ka+2})^2}\int_{\R^n}(\wu(y,-2\ln\rho)-\wu(y,-2\ln r))^2e^{-|y|^2}\,dy,
    \end{multline}
    where $\wu=\wu_\ka$ is as in \eqref{eq:chan-coor}.
    \item When $\rho=0$, for $W_{\ka,\al,\e,\de}=W_{\ka,\al,\e,\de,0}$ and $0<r<r_0=r_0(\ka_0,\al,\e)=\frac{r_0(\al,\e)}{\ka_0^{2/\e}}$,
    \begin{multline}
        \label{eq:par-weiss-deriv-2}
        \frac{d}{dr}W_{\ka,\al,\e,\de}(r,u,z_0)\\
        \ge \frac\ka{2r^{2\ka+3-\e/2}}\left|\int_{S_r}(\ka u-(x-x_0)\cdot\D u-2(t-t_0)\pa_tu)uG_{z_0}\right|.
    \end{multline}
\end{enumerate}
\end{theorem}

Although we work with $W_{\ka,\al,\e,\delta}$ throughout most of this paper, the monotonicity of $W_{\ka,\al,\e,\delta,\rho}$ will be used when we establish the rotation estimate in Lemma~\ref{lem:par-rot-est}.

\begin{proof}
The proof is divided into several steps.

\medskip\noindent\emph{Step 1.} Without loss of generality, we assume $z_0=0$. We write for simplicity $W_{\ka,\rho}=W_{\ka,\al,\e,\de,\rho}$. Let $w$ be the homogeneous replacement as in \eqref{eq:hom-rep}.
Note that we can write
$$
W_{\ka,\rho}(r,u)=\Phi(r)\int_{S_r\sm S_\rho}(-2t)|\D u|^2G-\Psi(r)\int_{S_r\sm S_\rho}\ka u^2G+\Phi(r)\|u\|_{\fso}^2e^{-\frac1{r}}r^{-\de},
$$
where $\Phi(r)=\frac{e^{ar^\al}}{r^{2\ka+2}-\rho^{2\ka+2}}$ and $\Psi(r)=\frac{(1-br^\e)e^{ar^\al}}{r^{2\ka+2}-\rho^{2\ka+2}}$ are as in Lemma~\ref{lem:par-weiss-lem}. Then, by using \eqref{eq:par-alm-min-def-kappa} and \eqref{eq:par-weiss-lem-1}, we deduce 
\begin{align*}
    &\frac d{dr}W_{\ka,\rho}(r,u)\\
    &=\Phi'(r)\int_{S_r\sm S_\rho}(-2t)|\D u|^2G-\Psi'(r)\int_{S_r\sm S_\rho}\ka u^2G+2r\Phi(r)\int_{\R^n\times\{-r^2\}}2r^2|\D u|^2G\\
    &\qquad-2r\Psi(r)\int_{\R^n\times\{-r^2\}}\ka u^2G+\|u\|_{\fso}^2\frac{d}{dr}\left(\Phi(r)e^{-\frac1r}r^{-\de}\right)\\
    &=\frac{\Phi'(r)}{1-r^\al}\int_{S_r\sm S_\rho}((1-r^\al)(-2t)|\D u|^2-\ka u^2)G+\left(\frac{\Phi'(r)}{1-r^\al}-\Psi'(r)\right)\int_{S_r\sm S_\rho}\ka u^2G\\
    &\qquad+2r\Phi(r)\int_{\R^n\times\{-r^2\}}2r^2|\D u|^2G-2r\Psi(r)\int_{\R^n\times\{-r^2\}}\ka u^2G\\
    &\qquad+\|u\|_{\fso}^2\frac{d}{dr}\left(\Phi(r)e^{-\frac1r}r^{-\de}\right)\\
    & \begin{multlined}\ge\frac{\Phi'(r)}{1-r^\al}\bigg(\int_{S_r\sm S_\rho}\big[(1+r^\al)(-2t)|\D w|^2-\ka w^2-2(\ka u-x\cdot \D u-2t\pa_tu)(u-w)\\
    +(\ka+2r^\al)(u-w)^2\big]G+2\|u\|_{\fso}^2e^{-\frac1r}\bigg)\end{multlined}\\
    &\qquad+\left(\frac{\Phi'(r)}{1-r^\al}-\Psi'(r)\right)\int_{S_r\sm S_\rho}\ka u^2G+2r\Phi(r)\int_{\R^n\times\{-r^2\}}2r^2|\D u|^2G\\
    &\qquad-2r\Psi(r)\int_{\R^n\times\{-r^2\}}\ka u^2G+\|u\|_{\fso}^2\frac{d}{dr}\left(\Phi(r)e^{-\frac1r}r^{-\de}\right)\\
    &=I+II+III+IV+V,
\end{align*}
where \begin{align*}
    &I=\frac{\Phi'(r)(1+r^\al)}{1-r^\al}\int_{S_r\sm S_\rho}(-2t)|\D w|^2G+2r\Phi(r)\int_{\R^n\times\{-r^2\}}2r^2|\D u|^2G,\\
    &II=-\frac{\Phi'(r)}{1-r^\al}\int_{S_r\sm S_\rho}\ka w^2G-2r\Psi(r)\int_{\R^n\times\{-r^2\}}\ka u^2G,\\
    &III=\frac{\Phi'(r)}{1-r^\al}\int_{S_r\sm S_\rho}\left[-2(\ka u-x\cdot \D u-2t\pa_tu)(u-w)+(\ka+2r^\al)(u-w)^2\right]G,\\
    & IV= \left(\frac{\Phi'(r)}{1-r^\al}-\Psi'(r)\right)\int_{S_r\sm S_\rho}\ka u^2G,\\
    &V=\frac{2\Phi'(r)\|u\|^2e^{-\frac1r}}{1-r^\al}+\|u\|_{\fso}^2\frac{d}{dr}\left(\Phi(r)e^{-\frac1r}r^{-\de}\right).
\end{align*}

\medskip\noindent \emph{Step 2.} In this step, we estimate the terms $I$-$V$. We begin with $I$ and $II$. By using the homogeneity of $w$, we can directly compute
\begin{align*}
    \int_{S_r\sm S_\rho}(-2t)|\D w|^2G&=\frac{r^{2\ka+2}-\rho^{2\ka+2}}{(\ka+1)r^{2\ka}}\int_{\R^n\times\{-r^2\}}2r^2|\D u|^2G\,dx,\\
    \int_{S_r\sm S_\rho}w^2G&=\frac{r^{2\ka+2}-\rho^{2\ka+2}}{(\ka+1)r^{2\ka}}\int_{\R^n\times\{-r^2\}}u^2G\,dx.
\end{align*}
Combining these equalities with \eqref{eq:par-weiss-lem-3} and \eqref{eq:par-weiss-lem-4}, we obtain
\begin{align*}
    I=\left(\frac{\Phi'(r)(1+r^\al)}{1-r^\al}+2r\Phi(r)\frac{(\ka+1)r^{2\ka}}{r^{2\ka+2}-\rho^{2\ka+2}}\right)\int_{S_r\sm S_\rho}(-2t)|\D w|^2G\ge 0,
\end{align*}
and \begin{align*}
    II&=\left(-\frac{\Phi'(r)}{1-r^\al}-2r\Psi(r)\frac{(\ka+1)r^{2\ka}}{r^{2\ka+2}-\rho^{2\ka+2}}\right)\int_{S_r\sm S_\rho}\ka w^2G\\
    &\ge \frac{\left(2\ka+2-\e/8\right)b\Phi(r)r^{2\ka+1+\e}}{r^{2\ka+2}-\rho^{2\ka+2}}\int_{S_r\sm S_\rho}\ka w^2G.
\end{align*}

Next, we estimate $III$. Note that \begin{align*}
    \Phi'(r)&=(a\al r^\al(r^{2\ka+2}-\rho^{2\ka+2})-(2\ka+2)r^{2\ka+2})\frac{\Phi(r)}{r(r^{2\ka+2}-\rho^{2\ka+2})}\\
    &\le (a\al r^\al-(2\ka+2))\frac{\Phi(r)r^{2\ka+1}}{r^{2\ka+2}-\rho^{2\ka+2}}\\
    &\le -(2\ka+1)\frac{\Phi(r)r^{2\ka+1}}{r^{2\ka+2}-\rho^{2\ka+2}},\quad r<\frac{1}{(16\ka_0)^{1/\al}}.
\end{align*}
This, along with \eqref{eq:par-(u-w)-transf-1}, produces
\begin{align*}
    III
    &=-\frac{\Phi'(r)}{1-r^\al}\bigg(\frac{2\rho^{2\ka+2}}{\pi^{n/2}}\int_{\R^n}(\wu(y,-2\ln \rho)-\wu(y,-2\ln r))^2e^{-|y|^2}\,dy\\
    &\qquad\qquad\qquad+(\ka+2-2r^\al)\int_{S_r\sm S_\rho}(u-w)^2G\bigg)\\
    &\ge \frac{(4\ka+2)\Phi(r)r^{2\ka+1}\rho^{2\ka+2}}{\pi^{n/2}(r^{2\ka+2}-\rho^{2\ka+2})}\int_{\R^n}(\wu(y,-2\ln \rho)-\wu(y,-2\ln r))^2e^{-|y|^2}\,dy\\
    &\qquad+\frac{(2\ka+1)(\ka+1)\Phi(r)r^{2\ka+1}}{r^{2\ka+2}-\rho^{2\ka+2}}\int_{S_r\sm S_\rho}(u-w)^2G.
\end{align*}

Regarding $IV$, we simply use \eqref{eq:par-weiss-lem-2} to get 
\begin{align*}
    IV\ge -\frac{\left(2\ka+2-\e/4\right)b\Phi(r)r^{2\ka+1+\e}}{r^{2\ka+2}-\rho^{2\ka+2}}\int_{S_r\sm S_\rho}\ka u^2G.
\end{align*}

Finally, to deal with $V$, we recall the inequality $r^{2\ka+2}-\rho^{2\ka+2}\ge \frac12r^{2\ka+2}$ to get
\begin{align*}
    \Phi'(r)=\left(a\al r^\al-\frac{(2\ka+2)r^{2\ka+2}}{r^{2\ka+2}-\rho^{2\ka+2}}\right)\frac{\Phi(r)}r\ge \left(a\al r^\al-4(\ka+1)\right)\frac{\Phi(r)}r,
\end{align*}
which yields
\begin{align*}
    \frac{V}{\|u\|_{\fso}^2}&=\frac{2e^{-\frac1r}}{1-r^\al}\Phi'(r)+\Phi'(r)e^{-\frac1r}r^{-\de}+\left(\frac1r-\de\right)\frac{\Phi(r)}re^{-\frac1r}r^{-\de}\\
    &\ge \left(\left(\frac2{1-r^\al}+r^{-\de}\right)\left(a\al r^\al-4(\ka+1)\right)+\left(\frac1r-\de\right)r^{-\de}\right)\frac{\Phi(r)}re^{-\frac1r}\\
    &\ge 0,\quad 0<r<\frac{r_0(\e)}{\ka_0}.
\end{align*}

\medskip\noindent\emph{Step 3.} By combining the results in Step 1 and Step 2, we get 
\begin{align*}
    &\frac{r^{2\ka+2}-\rho^{2\ka+2}}{\Phi(r)r^{2\ka+1}}\frac{d}{dr}W_{\ka,\rho}(r,u)\\
    &\ge \left(2\ka+2-\e/8\right)\ka br^\e\int_{S_r\sm S_\rho}w^2G+(2\ka+1)(\ka+1)\int_{S_r\sm S_\rho}(u-w)^2G\\
    &\qquad-\left(2\ka+2-\e/4\right)\ka br^\e\int_{S_r\sm S_\rho}u^2G\\
    &\qquad+\frac{(4\ka+2)\rho^{2\ka+2}}{\pi^{n/2}}\int_{\R^n}(\wu(y,-2\ln\rho)-\wu(y,-2\ln r))^2e^{-|y|^2}\,dy.
\end{align*}
On the other hand, we take $\mu=\frac{\e}{12\left(2\ka+2-\e/4\right)}$, which is tailor-made to satisfy $1+\mu=\frac{2\ka+2-\e/6}{2\ka+2-\e/4}$, and apply Young's inequality to have that for $0\le\rho<r<\frac{r_0(\al,\e)}{\ka_0^{2/\e}}$, 
\begin{align*}
    &\left(2\ka+2-\e/8\right)\ka br^\e\int_{S_r\sm S_\rho}w^2G+(2\ka+1)(\ka+1)\int_{S_r\sm S_\rho}(u-w)^2G\\
    &\qquad-\left(2\ka+2-\e/4\right)\ka br^\e\int_{S_r\sm S_\rho}u^2G\\
    &\ge \left(2\ka+2-\e/8\right)\ka br^\e\int_{S_r\sm S_\rho}w^2G+(2\ka+1)(\ka+1)\int_{S_r\sm S_\rho}(u-w)^2G\\
    &\qquad-\left(2\ka+2-\e/4\right)\ka br^\e\left((1+\mu)\int_{S_r\sm S_\rho}w^2G+(1+1/\mu)\int_{S_r\sm S_\rho}(u-w)^2G\right)\\
    &=\frac1{24}\e\ka br^\e\int_{S_r\sm S_\rho}w^2G\\
    &\qquad+\left((2\ka+1)(\ka+1)-\left(2\ka+2-\e/4\right)\ka br^\e(1+1/\mu)\right)\int_{S_r\sm S_\rho}(u-w)^2G\\
    &\ge 5\ka(\ka+1)r^\e\int_{S_r\sm S_\rho}w^2G+\ka(\ka+1)\int_{S_r\sm S_\rho}(u-w)^2G.
\end{align*}
By combining the precious two inequalities, we deduce
\begin{multline}
    \label{eq:par-weiss-deriv-est}
    \frac{r^{2\ka+2}-\rho^{2\ka+2}}{\Phi(r)r^{2\ka+1}}\frac{d}{dr}W_{\ka,\rho}(r,u)\\
    \begin{multlined}\ge  5\ka(\ka+1)r^\e\int_{S_r\sm S_\rho}w^2G+\ka(\ka+1)\int_{S_r\sm S_\rho}(u-w)^2G\\
    +\frac{(4\ka+2)\rho^{2\ka+2}}{\pi^{n/2}}\int_{\R^n}(\wu(y,-2\ln\rho)-\wu(y,-2\ln r))^2e^{-|y|^2}\,dy.\end{multlined}
\end{multline}
This gives \eqref{eq:par-weiss-deriv-1}.

\medskip\noindent\emph{Step 4.} The purpose of this step is to obtain \eqref{eq:par-weiss-deriv-2}. To this aim, we let $\rho=0$, and observe that $(\wu-\ww)\wu=0$ on $\R^n\times\{-2\ln r\}$ and that for any $y\in\R^n$
$$
\lim_{\tau\to \infty}\left|(\wu(y,\tau)-\ww(y,\tau))\wu(y,\tau)e^{-(\ka+1)\tau}\right|\le \lim_{\tau\to\infty}\left(2\|u\|^2_{L^\infty(S_1)}e^{-\tau}\right)=0.
$$
It then follows that
\begin{align*}
        &\int_{S_r}(\ka u-x\cdot\D u-2t\pa_tu)uG\\
        &=\frac2{\pi^{n/2}}\int_{\R^n\times(-2\ln r,\infty)}(\pa_\tau(\wu-\ww))\wu e^{-|y|^2}e^{-(\ka+1)\tau}\,dyd\tau\\
        &=-\frac2{\pi^{n/2}}\int_{\R^n\times(-2\ln r,\infty)}(\wu-\ww)(\pa_\tau\wu) e^{-|y|^2}e^{-(\ka+1)\tau}\,dyd\tau\\
        &\qquad+\frac{2(\ka+1)}{\pi^{n/2}}\int_{\R^n\times(-2\ln r,\infty)}(\wu-\ww)\wu e^{-|y|^2}e^{-(\ka+1)\tau}\,dyd\tau\\
        &=-\int_{S_r}(u-w)(\ka u-x\cdot\D u-2t\pa_tu)G\,dxdt+2(\ka+1)\int_{S_r}(u-w)uG\,dxdt\\
        &=(\ka+1)\int_{S_r}(u-w)^2G+2(\ka+1)\int_{S_r}(u-w)wG,
\end{align*}
where we used \eqref{eq:hom-rep-time-deriv} in the first step and \eqref{eq:par-(u-w)-transf-2} in the last equality. Thus \begin{align*}
        \left|\int_{S_r}(\ka u-x\cdot\D u-2t\pa_tu)uG\right|\le 2(\ka+1)r^{-\e/2}\int_{S_r}(u-w)^2G+(\ka+1)r^{\e/2}\int_{S_r}w^2G.
\end{align*}
Therefore, by combining this with \eqref{eq:par-weiss-deriv-est}, we conclude that \begin{align*}
    \frac\ka2r^{\e/2}\left|\int_{S_r}(\ka u-x\cdot\D u-2t\pa_tu)uG\right|&\le \ka(\ka+1)\int_{S_r}(u-w)^2G+\frac{\ka(\ka+1)}2r^\e\int_{S_r}w^2G\\
    &\le \frac{r}{\Phi(r)}\frac{d}{dr}W_\ka(r,u).
\end{align*}
This implies \eqref{eq:par-weiss-deriv-2}.
\end{proof}

Next, we deal with the Almgren-type frequency in the parabolic setting. Poon proved in \cite{Poo96} that if $u$ is a caloric function in $S_1$ (i.e., $\Delta u-\partial_tu=0$ in $S_1$), then its caloric frequency
$$
N(r,u,z_0):=\frac{r^2\int_{\R^n\times\{t_0-r^2\}}|\D u|^2G_{z_0}dx}{\int_{\R^n\times\{t_0-r^2\}}u^2G_{z_0}dx}
$$
is monotone nondecreasing in $r\in(0,1)$. Concerning the parabolic Signorini problem, \cite{DanGarPetTo17} considered its averaged version
$$
N^0(r,u,z_0):=\frac{\int_{S_r(t_0)}2(t_0-t)|\D u|^2G_{z_0}}{\int_{S_r(t_0)}u^2G_{z_0}},
$$
and proved the generalized frequency formula related to $N^0$ when $z_0$ is a free boundary point. For almost minimizers, we need some modifications on $N^0$. Given free boundary point $z_0\in \Gamma(u)\cap Q_{1/2}'$, we let
$$
 N_\de(r,u,z_0):=\frac{\int_{S_r(t_0)}2(t_0-t)|\D u|^2G_{z_0}+\|u\|_{\fsz}^2e^{-\frac1{r}}r^{-\de}}{\int_{S_r(t_0)}u^2G_{z_0}}.
$$
We then define the multiplicative modification of $N_\de$
$$
\N_{\ka_0,\e,\de}(r,u,z_0):=\frac1{1-br^\e}N_\de(r,u,z_0),
$$
where $b$ is as in Theorem~\ref{thm:par-weiss} (or Lemma~\ref{lem:par-weiss-lem}), as well as the truncation of $\tilde N_{\ka_0,\e,\de}$
$$
\widehat{N}_{\ka_0,\e,\de}(r,u,z_0):=\min\{\N_{\ka_0,\e,\de}(r,u,z_0),\ka_0\},\quad 0<r<r_0=r_0(\ka_0,\e)=\frac{r_0(\e)}{\ka_0^{2/\e}}.
$$
When $z_0=0$, we simply write $N^0(r,u)$, $N_\de(r,u)$, etc.

As demonstrated in \cite{JeoPet21}*{Theorem~5.4}, the monotonicity of $W_{\ka,\al,\e,\delta}$ readily implies that of the truncated frequency $\widehat{N}_{\ka_0,\e,\de}$.

\begin{theorem}[Almgren-type monotonicity formula]\label{thm:par-almgren}
Let $u,z_0,\ka_0,\delta,\e$ be as in Theorem~\ref{thm:par-weiss}. Then $\widehat{N}_{\ka_0,\e,\de}(r,u,z_0)$ is nondecreasing in $0<r<r_0=r_0(\ka_0,\e)=\frac{r_0(\e)}{\ka_0^{2/\e}}$.
\end{theorem}

\begin{proof}
We may assume without loss of generality $z_0=0$. Take $r_0=r_0(\ka_0,\e)$ small so that $1-br^\e>0$. If $\widehat{N}_{\ka_0,\e,\de}(r,u)<\ka$ for some $r\in(0,r_0)$ and $\ka\in(0,\ka_0)$, then $$
W_{\ka,\al,\e,\de}(r,u)=\frac{e^{ar^\al}}{r^{2\ka+2}}(1-br^\e)\left(\int_{S_r}u^2G\right)(\N_{\ka_0,\e,\de}(r,u)-\ka)<0.
$$
For any $0<s<r$, we have by Theorem~\ref{thm:par-weiss} that $W_{\ka,\al,\e,\de}(s,u)\le W_{\ka,\al,\e,\de}(r,u)<0$, and hence $\widehat{N}_{\ka_0,\e,\de}(s,u)<\ka$, as desired.
\end{proof}


\section{Almgren rescalings and blowups}
The main objective of this section is to derive the proper lower bound for the frequency for almost minimizers at free boundary points. For this purpose, we deal with so-called Almgren blowups, which become global solutions of the parabolic Signorini problem. It is known that even-symmetric (in $x_n$-variable) solutions possess the minimal frequency of $3/2$.

In the study of the Signorini problem (both in elliptic and parabolic settings), the even symmetry of the solution with respect to the thin space is imperative. The symmetry ensures that the growth rate of the solution over the ``thick'' strip $S_r(z_0)$ match that over the ``thin'' strip $S'_r(z_0)$. This allows us to extract the information about the behavior of solutions on the thin space using the Almgren-type monotonicity formula.

In the case of solutions of the parabolic Signorini problem, the symmetry assumption is not resrictive, because if $u$ is a solution then its even symmetrization
$$
u^*(x',x_n,t)=\frac{u(x',x_n,t)+u(x',-x_n,t)}2
$$
is still a solution. However, this property is not available for almost minimizers, as the even symmetrization can disrupt the almost Signorini property, even in the time-independent case (see \cite{JeoPetSVG24}*{Example~6.1}).

Therefore, in the remainder of this paper, we assume that the almost minimizer $u$ is even symmetric in $x_n$-variable.

\medskip
Next, we introduce another type of competitor for $u\in\fs_{z_0}$ aside from homogeneous replacement. We say that $v$ is a \emph{parabolic Signorini replacement} of $u$ in $S_r(t_0)$ if $v$ is the solution of a parabolic Signorini problem in $S_r(t_0)$ with $v=u$ on $\R^n\times\{t_0-r^2\}$ and $v-u\in L^2(t_0-r^2,t_0;W^{1,2}_0(\R^n,G_{z_0}))$.

We remark that the regularity assumption on $u\in\fs_{z_0}$ is not sufficient to ensure the existence of its parabolic Signorini replacement. To rectify this issue, we consider convolutions with mollifiers. For a standard mollifier $\vp=\vp(x)$ in $\R^n$ and a small constant $\mu>0$, we let $\vp_\mu(x):=(1/\mu)^n\vp(x/\mu)$. We set
\begin{align}
    \label{eq:conv}
    u_\mu(x,t):=u*\vp_\mu(x,t),\quad (x,t)\in S_1.
\end{align}
Then $u_\mu(\cdot,-r^2)\in W^2_\infty(\R^n)$ for a.e. $r\in(0,1)$ and $\|u_\mu-u\|_{\fs_{z_0}}\to0$ as $\mu\to0$. By Theorem~\ref{thm:exist-weak-sol}, for such $r$, there exists a unique parabolic Signorini replacement of $u_\mu$ in $S_r$.

\begin{remark}
    \label{rmk:conv-alm-min}
    $u_\mu$ satisfies the almost parabolic Signorini property in $S_r$, $0<r<1$, with a gauge function $\eta(r)=r^{\al/2}$ and additional additive error $C(n,\al)\|u-u_\mu\|^2_{\fs_{z_0}}$. 

    Indeed, we assume without loss of generality $z_0=0$. Since $v:=u-u_\mu+v_\mu$ is a valid competitor of $u$ in $S_r$, we have by \eqref{eq:par-alm-min-def} and Young's inequality
    \begin{align*}
    &\int_{S_r}(1-r^{\al/2})(-t)|\D u_\mu|^2G+(x\cdot\D u_\mu-2t\pa_tu_\mu)(u_\mu-v_\mu)G\\
    &=\int_{S_r}(1-r^{\al/2})(-t)|\D u+\D(u_\mu-u)|^2G+(-x\cdot\D u-2t\pa_tu)(u-v)G\\
    &\qquad\qquad+(-x\cdot\D(u_\mu-u)-2t\pa_t(u_\mu-u))(u_\mu-v_\mu)G\\
    &\le \int_{S_r}(1-r^\al)(-t)|\D u|^2G+(-x\cdot\D u-2t\pa_tu)(u-v)G\\
    &\qquad+C(r)\|u-u_\mu\|^2_{\fs_0}+r^\al\int_{S_r}(u_\mu-v_\mu)^2G\\
    &\le (1+r^\al)\int_{S_r}(-t)|\D v|^2G+r^\al\int_{S_r}(u-v)^2G+\|u\|_{\fso}^2e^{-\frac1r}\\
    &\qquad+C(r)\|u-u_\mu\|_{\fs_0}^2+r^\al\int_{S_r}(u_\mu-v_\mu)^2G\\
    &\le (1+r^{\al/2})\int_{S_r}(-t)|\D v_\mu|^2G+2r^{\al/2}\int_{S_r}(u_\mu-v_\mu)^2G+\|u_\mu\|_{\fso}^2e^{-\frac1r}\\
    &\qquad+C(r)\|u-u_\mu\|_{\fs_0}^2.
\end{align*}
\end{remark}

\begin{lemma}\label{lem:par-sig-alm-min-diff-est-weight}
Let $u\in \fsz$ satisfy the almost parabolic Signorini property at $z_0\in Q'_{1/2}$. Suppose that $u$ has a parabolic Signorini replacement $v$ in $S_r(t_0)$. Then there exist constants $r_0>0$ and $C>0$, depending only on $\al$, such that if $0<r<r_0$, 
\begin{align}
    \label{eq:par-sig-alm-min-grad-diff-est-weight} \int_{S_r(t_0)}(t_0-t)|\D(u-v)|^2G_{z_0}&\le Cr^\al\int_{S_r(t_0)}(t_0-t)|\D u|^2G_{z_0}+C\|u\|_{\fsz}^2e^{-\frac1r},\\
    \label{eq:par-sig-alm-min-diff-est-weight}\int_{S_r(t_0)}(u-v)^2G_{z_0}&\le Cr^\al\int_{S_r(t_0)}(t_0-t)|\D u|^2G_{z_0}+C\|u\|_{\fsz}^2e^{-\frac1r}.
\end{align}
\end{lemma}

\begin{proof} 
We may assume without loss of generality $z_0=0$. By the variational inequality of $v$, we have $$
\int_{S_r}(-2t)\D v\D(v-u)G+(-x\cdot\D v-2t\pa_tv)(v-u)G\le 0.
$$
This, combined with the almost Signorini property of $u$ (equation \eqref{eq:par-alm-min-def}), gives
\begin{align*}
    &\int_{S_r}(-t)|\D(u-v)|^2G\\
    &=\int_{S_r}(-t)|\D u|^2G-\int_{S_r}(-t)|\D v|^2G+2\int_{S_r}(-t)\D(v-u)\D vG\\
    &\le r^\al\int_{S_r}(-t)(|\D u|^2+|\D v|^2)G+\int_{S_r}(-x\cdot\D u-2t\pa_tu)(v-u)G\\
    &\qquad+r^\al\int_{S_r}(u-v)^2G+\|u\|_{\fso}^2e^{-\frac1r}-\int_{S_r}(-x\cdot\D v-2t\pa_tv)(v-u)G\\
    &=\int_{S_r}(-x\cdot\D(u-v)-2t\pa_t(u-v))(v-u)G+r^\al\int_{S_r}(-t)(|\D u|^2+|\D v|^2)G\\
    &\qquad+r^\al\int_{S_r}(u-v)^2G+\|u\|_{\fso}^2e^{-\frac1r}.
\end{align*}
To compute the first term in the last line, we consider $\wu(y,\tau):=u\left(2e^{-\frac\tau2}y,-e^{-\tau}\right)$ and $\wv(y,\tau):=v\left(2e^{-\frac\tau2}y,-e^{-\tau}\right)$, which correspond to \eqref{eq:chan-coor} with $\ka=0$. Since $\wu-\wv=0$ on $\R^n\times\{-2\ln r\}$, we have by Integration by parts
\begin{align}\begin{split}\label{eq:par-(u-v)-transf}
    &\int_{S_r}(-x\cdot\D(u-v)-2t\pa_t(u-v))(v-u)G\\
    &=\frac2{\pi^{n/2}}\int_{\R^n\times(-2\ln r,\infty)}(\pa_\tau(\wu-\wv))(\wv-\wu)e^{-|y|^2}e^{-\tau}\,dyd\tau\\
    &=-\frac1{\pi^{n/2}}\int_{\R^n\times(-2\ln r,\infty)}\pa_\tau\left((\wu-\wv)^2\right)e^{-|y|^2}e^{-\tau}\,dyd\tau\\
    &\le -\frac1{\pi^{n/2}}\int_{\R^n\times(-2\ln r,\infty)}(\wu-\wv)^2e^{-|y|^2}e^{-\tau}\,dyd\tau\\
    &=-\int_{S_r}(u-v)^2G\,dxdt.
\end{split}\end{align}
It then follows that 
\begin{align}
    \label{eq:par-grad-u-v-diff-est}
    \int_{S_r}(-t)|\D(u-v)|^2\le r^\al\int_{S_r}(-t)(|\D u|^2+|\D v|^2)G+\|u\|_{\fso}^2e^{-\frac1r}.
\end{align}
This gives \begin{align*}
    \int_{S_r}(-t)|\D v|^2G&\le 2\int_{S_r}(-t)|\D u|^2G+2\int_{S_r}(-t)|\D(u-v)|^2G\\
    &\le 4\int_{S_r}(-t)|\D u|^2G+2r^\al\int_{S_r}(-t)|\D v|^2G+2\|u\|_{\fso}^2e^{-\frac1r},
\end{align*}
which implies \begin{align}
    \label{eq:par-grad-u-v-est}
    \int_{S_r}(-t)|\D v|^2G\le C\int_{S_r}(-t)|\D u|^2+C\|u\|_{\fso}^2e^{-\frac1r},\quad r<r(\al).
\end{align}
By combining this with \eqref{eq:par-grad-u-v-diff-est}, we obtain \eqref{eq:par-sig-alm-min-grad-diff-est-weight}.

\medskip Regarding \eqref{eq:par-sig-alm-min-diff-est-weight}, we use the almost parabolic Signorini property of $u$ and the parabolic Signorini property of $v$ (i.e., equations \eqref{eq:par-alm-min-def} and \eqref{eq:par-Sig}) to have 
\begin{align*}
    &\int_{S_r}(-t)|\D u|^2G+\int_{S_r}(-x\cdot\D u-2t\pa_tu)(u-v)G\\
    &\le\int_{S_r}(-t)|\D v|^2G+r^\al\int_{S_r}(u-v)^2G+\|u\|_{\fso}^2e^{-\frac1r}+r^\al\int_{S_r}(-t)(|\D u|^2+|\D v|^2)G\\
    &\le \int_{S_r}(-t)|\D u|^2G-\int_{S_r}(-x\cdot \D v-2t\pa_tv)(v-u)G\\
    &\qquad +r^\al\int_{S_r}(u-v)^2G+\|u\|_{\fso}^2e^{-\frac1r}+r^\al\int_{S_r}(-t)(|\D u|^2+|\D v|^2)G,
\end{align*}
and thus \begin{align*}
    &\int_{S_r}(-x\cdot\D(u-v)-2t\pa_t(u-v))(u-v)G\\
    &\qquad\le r^\al\int_{S_r}(u-v)^2G+\|u\|_{\fso}^2e^{-\frac1r}+r^\al\int_{S_r}(-t)(|\D u|^2+|\D v|^2)G.
\end{align*} 
This, together with \eqref{eq:par-(u-v)-transf}, gives
$$
\int_{S_r}(u-v)^2G\le Cr^\al\int_{S_r}(-t)(|\D u|^2+|\D v|^2)G+C\|u\|_{\fso}^2e^{-\frac1r}.
$$
Finally, by combining this and \eqref{eq:par-grad-u-v-est}, we conclude \eqref{eq:par-sig-alm-min-diff-est-weight}.
\end{proof}

\begin{corollary}
    \label{cor:conv-diff-est}
    Suppose that $u\in \fs_{z_0}$ satisfies the almost Signorini property at $z_0\in Q'_{1/2}$. For $u_\mu$ be as in \eqref{eq:conv}, let $v_\mu$ be the parabolic Signorini replacement of $u_\mu$ in $S_r(t_0)$. Then there are $r_0>0$ and $C>0$, depending only on $\al$, such that for $0<r<r_0$
    \begin{align*}
        &\begin{multlined}\int_{S_r(t_0)}(t_0-t)|\D(u_\mu-v_\mu)|^2G_{z_0}\le Cr^{\al/2}\int_{S_r(t_0)}|\D u_\mu|^2G_{z_0}+C\|u_\mu\|_{\fs_{z_0}}^2e^{-\frac1r}\\+C_0(n,\al)\|u-u_\mu\|^2_{\fs_{z_0}},\end{multlined}\\
        &\begin{multlined}\int_{S_r(t_0)}(u_\mu-v_\mu)^2G_{z_0}\le Cr^{\al/2}\int_{S_r(t_0)}|\D u_\mu|^2G_{z_0}+C\|u_\mu\|_{\fs_{z_0}}^2e^{-\frac1r}\\+C_0(n,\al)\|u-u_\mu\|^2_{\fs_{z_0}}.\end{multlined}
    \end{align*}
\end{corollary}

\begin{proof}
    For the proof, we can use the almost parabolic Signorini property of $u_\mu$ (Remark~\ref{rmk:conv-alm-min}) and follow the argument in Lemma~\ref{lem:par-sig-alm-min-diff-est-weight}.
\end{proof}

For $z_0=(x_0,t_0)\in\Ga(u) \cap Q'_{1/2}$ and $0<r<1/2$, consider the \emph{Almgren rescaling} of $u$ at $z_0$ $$
u^A_{z_0,r}(x,t):=\frac{u(rx+x_0,r^2t+t_0)}{\left(\frac1{r^2}\int_{S_r(t_0)}u^2G_{z_0}\right)^{1/2}},\quad (x,t)\in Q_{1/(2r)}.
$$
It satisfies the normalization and scaling properties
\begin{align*}
&\int_{S_1}(u_{z_0,r}^A)^2G=1,\\
&N^0(\rho,u^A_{z_0,r},0)=N^0(\rho r,u,z_0),\quad \rho<1/(2r).
\end{align*}
We will call the limits of $u_{z_0,r}^A$ over any sequence $r=r_j\to 0+$ \emph{Almgren blowups} of $u$ at $z_0$, denoted by $u_{z_0,0}^A$. When $z_0=0$, we simply write $u_r^A=u_{0,r}^A$ and $u_0^A=u_{0,0}^A$.

\begin{proposition}[Existence of Almgren blowups]\label{prop:par-exist-Alm-blowup}
Let $z_0\in Q'_{1/2}\cap \Ga(u)$ be such that $\widehat{N}_{\ka_0,\de}(0+,u,z_0)=\ka<\ka_0$ for some $0<\de<2$ and $\ka_0>2$. Then every sequence of Almgren rescalings $u_{z_0,r_j}^A$, with $r_j\to 0+$, contains a subsequence, still denoted by $r_j$, such that for a function $u^A_{x_0,0}\in W^{1,0}_{2,\loc}(S_1,G)\cap C^{1,0}_{\loc}(Q_1^\pm\cup Q_1')$
\begin{align*}
    &u^A_{z_0,r_j}\ra u^A_{z_0,0}\quad\text{in } W^{1,0}_{2,\loc}(S_1,G),\\
    &u^A_{z_0,r_j}\ra u^A_{z_0,0}\quad\text{in } C^{1,0}_{\loc}(Q_1^\pm\cup Q_1').
\end{align*}
Moreover, $u^A_{z_0,0}$ is a nonzero solution of the parabolic Signorini problem in $S_1$, even in $x_n$.
\end{proposition}

\begin{proof}
We may assume without loss of generality $z_0=0$. From $\widehat{N}_{\ka_0,\de}(0+,u)=\ka<\ka_0$, we have $N_\de(0+,u)=\widehat{N}_{\ka_0,\de}(0+,u)=\ka$, and thus $N_\de(r_j,u)<\ka_0$ for small $r_j$. For such $r_j$, 
\begin{align}\label{eq:par-Alm-gradient-bound}
\int_{S_1}(-2t)|\D u^A_{r_j}|^2G=N^0(1,u_{r_j}^A)=N^0(r_j,u)\le N_\de(r_j,u)<\ka_0.
\end{align}
Combining this with $\int_{S_1}\left(u_{r_j}^A\right)^2G=1$, we see that $u_{r_j}^A$ is bounded in $W^{1,0}_2(S_1,G)$. Thus, there is a function $u_0^A\in W^{1,0}_{2,\loc}(S_1,G)$ such that up to a subsequence $$
u_{r_j}^A\ra u_0^A\quad\text{weakly in }W^{1,0}_{2,\loc}(S_1,G).
$$
Moreover, it is easy to see that each $u_{r_j}^A$ is an unweighted almost minimizers in $Q_{1/(2r_j)}$ with gauge function $\eta_{r_j}(\rho)=(r_j\rho)^\al\le \rho^\al$.
Thus, for small $\e>0$ and $K\Subset Q_1\cap\{t\le-\e\}$, we have by Theorem~\ref{thm:par-alm-min-grad-holder}
\begin{align*}
    \|u_{r_j}^A\|_{C^{\al,\al/2}(K)}+\|\D u^A_{r_j}\|_{C^{\be,\be/2}(K^\pm\cap K')}&\le C(n,K,\al,\e)\|u_{r_j}^A\|_{W^{1,0}_2(Q_1\cap\{t\le-\e\})}\\
    &\le C(n,K,\al,\e)\|u_{r_j}^A\|_{W^{1,0}_2(S_1,G)}.
\end{align*}
This, along with the boundedness of $\{u_{r_j}^A\}$ in $W^{1,0}_2(S_1,G)$, yields
$$
u^A_{r_j}\ra u_0^A\quad\text{in } C^{1,0}_{\loc}((Q_1^\pm\cup Q_1')\cap \{t\le-\e\}).
$$
Taking $\e\searrow 0$ and using Cantor's diagonal argument, we infer that over a subsequence $r=r_j\searrow 0$
$$
u^A_{r_j}\ra u_0^A\quad\text{in } C^{1,0}_{\loc}(Q_1^\pm\cup Q_1').
$$
Now, for each $r_j$, we take $\mu_j>0$ small so that the convolution $(u^A_{r_j})_{\mu_j}=u^A_{r_j}*\vp_{\mu_j}$ as in \eqref{eq:conv} satisfies $C_0(r_j,\al)\|(u_{r_j}^A)_{\mu_j}-u^A_{r_j}\|_{\fs_0}^2\to0$ as $r_j\to0$, where $C_0(r_j,\al)$ is as in Corollary~\ref{cor:conv-diff-est}. We then apply Theorem~\ref{thm:exist-weak-sol}. By considering $S_{1-\e}$ with small $\e>0$ if necessary, we may assume that there exists the parabolic Signorini replacement $(v_{r_J})_{\mu_j}$ of $(u_{r_j}^A)_{\mu_j}$ in $S_1$. By Corollary~\ref{cor:conv-diff-est}, \eqref{eq:par-Alm-gradient-bound} and the bound $N_\de(r_j,u)<\ka_0$, we have
\begin{align*}
    &\int_{S_1}(-t)|\D(u_{r_j}^A-(v_{r_j})_{\mu_j})|^2G\\
    &\le 2\int_{S_1}(-t)|\D((u_{r_j}^A)_{\mu_j}-(v_{r_j})_{\mu_j})|^2G+2\|(u_{r_j}^A)_{\mu_j}-u_{r_j}^A\|^2_{\fs_0}\\
    &\le Cr_j^{\al/2}\int_{S_1}(-t)|\D(u_{r_j}^A)_{\mu_j}|^2G+C\|(u_{r_j}^A)_{\mu_j}\|_{\fs_0}^2e^{-\frac1{r_j}}\\
    &\qquad+C(\al)C_0(r_j,\al)\|(u_{r_j}^A)_{\mu_j}-u^A_{r_j}\|_{\fs_0}^2\\
    &\le Cr_j^{\al/2}\int_{S_1}(-t)|\D u_{r_j}^A|^2G+C\|u_{r_j}^A\|_{\fs_0}^2e^{-\frac1{r_j}}+C(\al)C_0(r_j,\al)\|(u_{r_j}^A)_{\mu_j}-u^A_{r_j}\|_{\fs_0}^2\\
    &\le C\ka_0 r_j^{\al/2}+CN_\de(r_j,u)e^{-\frac1{r_j}}+C(\al)C_0(r_j,\al)\|(u_{r_j}^A)_{\mu_j}-u^A_{r_j}\|_{\fs_0}^2\\
    &\to 0\quad\text{as }r_j\to0+.
\end{align*}
Similarly, we can obtain
$$
\int_{S_1}(u_{r_j}^A-(v_{r_j})_{\mu_j})^2G\to 0 \quad\text{as }r_j\to0+.
$$
These estimates, combined with the bounded of $u_{r_j}^A$ in $W^{1,0}_2(S_1,G)$, implies that $(v_{r_j})_{\mu_j}$ is also bounded in $W^{1,0}_2(S_1,G)$ and
$$
u_{r_j}^A-(v_{r_j})_{\mu_j}\to 0\quad\text{strongly in } W^{1,0}_{2,\loc}(S_1,G),
$$
and hence
$$
(v_{r_j})_{\mu_j}\to u_0^A \quad\text{weakly in } W^{1,0}_{2,\loc}(S_1,G).
$$
By \cite{DanGarPetTo17}*{Theorem~7.3}, we have that $u_0^A$ is a solution of the parabolic Signorini problem in $S_1$ and that $(v_{r_j})_{\mu_j}$ is bounded in $W^{2,1}_2(S_1,G)$, which gives $(v_{r_j})_{\mu_j}^2G\to \left(u_0^A\right)^2G$ strongly in $L^1(S_1)$. This, combined with the observation
\begin{align*}
    2\int_{S_1}(v_{r_j})_{\mu_j}^2G\ge \int_{S_1}\left(u_{r_j}^A\right)^2G-2\int_{S_1}(u^A_{r_j}-(v_{r_j})_{\mu_j})^2G\to 1\quad\text{as } r_j\to 0+,
\end{align*}
produces $\int_{S_1}\left(u_0^A\right)^2G=\lim_{r_j\ra 0+}\int_{S_1}(v_{r_j})_{\mu_j}^2G\ge 1/2$, and hence $u_0^A\not\equiv 0$ in $S_1$. This completes the proof.
\end{proof}

In the subsequent lemma, we show that the additional exponential term present in $N_\delta$ is insignificant, as previously mentioned. This enables us to employ the results established in \cite{DanGarPetTo17} while proving a lower bound on Almgren's frequency for almost minimizers in Lemma~\ref{lem:par-min-freq}.

\begin{lemma}\label{lem:par-Almgren}
Suppose $u$ satisfies the almost parabolic Signorini property at $z_0\in \Gamma(u)\cap Q'_{1/2}$. If $\widehat{N}_{\ka_0,\de}(0+,u,z_0)=\ka<\ka_0$ for some $0<\de<2$ and $\ka_0>2$, then $$
\lim_{r\to 0}N^0(r,u,z_0)=\lim_{r\to 0}\frac{\int_{S_r(t_0)}2(t_0-t)|\D u|^2G_{z_0}}{\int_{S_r(t_0)}u^2G_{z_0}}=\ka.
$$
\end{lemma}

\begin{proof}
Without loss of generality, we may assume $z_0=0$. From $\widehat{N}_{\ka_0,\de}(0+,u)=\ka<\ka_0$, we have $N_{\de}(0+,u)=\ka$, thus it is enough to show 
$$
\lim_{r\to 0}\frac{\|u\|_{\fso}^2e^{-\frac1r}r^{-\de}}{\int_{S_r}u^2G}=0.
$$
To this aim, we assume to the contrary that 
$$
\limsup_{r\to 0}\frac{\|u\|_{\fso}^2e^{-\frac1r}r^{-\de}}{\int_{S_r}u^2G}=a_0\in (0,\ka].
$$
Then we have for some sequence $r=r_j\searrow 0$ 
$$
\frac{\|u\|_{\fso}^2e^{-\frac1{r_j}}r_j^{-\de}}{\int_{S_{r_j}}u^2G}>\frac{a_0}2\quad\text{ and }\quad \frac{\int_{S_{r_j}}(-2t)|\D u|^2G}{\int_{S_{r_j}}u^2G}>\ka-2a_0.
$$
Fix $\e\in(0,\al)$ and for $\de'=\de+\e/4$, consider $N_{\de'}(r,u)=\frac{\int_{S_r}(-2t)|\D u|^2G+\|u\|_{\fso}^2e^{-\frac1r}r^{-\de'}}{\int_{S_r}u^2G}$. Then for $r=r_j$, 
\begin{align*}
    N_{\de'}(r_j,u)=\frac{\int_{S_{r_j}}(-2t)|\D u|^2G}{\int_{S_{r_j}}u^2G}+\frac{\|u\|_{\fso}^2e^{-\frac1{r_j}}r_j^{-\de}}{\int_{S_{r_j}}u^2G}r_j^{-\e/4}>\frac{a_0}2r_j^{-\e/4}+\ka-2a_0.
\end{align*}
Thus, we have for any $\ka_1>\ka_0$
\begin{align*}
    \N_{\ka_1,\e,\de'}(r_j,u)\ge N_{\de'}(r_j,u)>\frac{a_0}2r_j^{-\e/4}+\ka-2a_0,
\end{align*}
and hence
\begin{align}
    \label{eq:par-trun-almgren-eq}
    \widehat{N}_{\ka_1,\de'}(0+,u)=\lim_{r_j\to 0}\min\{\N_{\ka_1,\e,\de'}(r_j,u),\ka_1\}=\ka_1.
\end{align}
On the other hand, since $N_{\de}(0+,u)=\ka$, there is $s_0>0$, independent of $\ka_1$, such that $N_{\de}(r,u)<2\ka$ for $r<s_0$. Moreover, by Theorem~\ref{thm:par-almgren}, there is a constant $c_0=c_0(\e)>0$ such that for any $\ka_1>\ka_0$ and $0<\de<2$, $\widehat{N}_{\ka_1,\e,\de}(r,u)$ is nondecreasing in $0<r<\frac{2c_0}{\ka_1^{2/\e}}$. By taking $c_0$ smaller if necessary, we may assume $c_0<s_0$. Note that if $\ka_1$ is sufficiently large, then $1-\frac{128(\ka_1+1)}\e\cdot \frac{c_0(\e)^\e}{\ka_1^2}>1/2$. Moreover, we have for $0<r<\frac{2c_0}{\ka_0^{2/\e}}$
\begin{align*}
    N_{\de'}(r,u)&=\frac{\int_{S_r}(-2t)|\D u|^2G}{\int_{S_r}u^2G}+\frac{\|u\|_{\fso}^2e^{-\frac1{r}}r^{-\de}}{\int_{S_r}u^2G}r^{-\e/4}\le N_\de(r,u)r^{-\e/4}\le 2\ka r^{-\e/4}.
\end{align*}
It then follows that
\begin{align*}
    \N_{\ka_1,\e,\de'}\left(\frac{c_0}{\ka_1^{2/\e}},u\right)&=\frac{1}{1-\frac{128(\ka_1+1)}\e\left(\frac{c_0}{\ka_1^{2/\e}}\right)^\e}N_{\de'}\left(\frac{c_0}{\ka_1^{2/\e}},u\right)\\
    &\le 4\ka\left(\frac{c_0}{\ka_1^{2/\e}}\right)^{-\e/4}\le C(\e,\ka,s_0)\ka_1^{1/2}.
\end{align*}
Therefore, we have for large $\ka_1$
$$
\widehat{N}_{\ka_1,\e,\de'}\left(\frac{c_0}{\ka_1^{2/\e}},u\right)\le C(\e,\ka,s_0)\ka_1^{1/2}<\ka_1.
$$
This, along with \eqref{eq:par-trun-almgren-eq}, contradicts the monotonicity of $\widehat{N}_{\ka_1,\e,\de'}$ in $\left(0,\frac{2c_0}{\ka_1^{2/\e}}\right)$.
\end{proof}

\begin{lemma}\label{lem:par-min-freq}
Let $u$ be an almost minimizer for the parabolic Signorini problem in $Q_1$ and $z_0\in \Gamma(u)\cap Q'_{1/2}$. If $\widehat{N}_{\ka_0,\de}(0+,u,z_0)=\ka<\ka_0$ for some $0<\de<2$ and $\ka_0>2$, then $$
\ka=3/2\quad\text{or}\quad \ka\ge 2.
$$
\end{lemma}

\begin{proof}
Without loss of generality, we may assume $z_0=0$. Let $u_0^A=\lim_{r_j\to 0}u_{r_j}^A$ be an Almgren blowup. Recall that it is a solution of the parabolic Signorini problem in $S_1$. From Lemma~\ref{lem:par-Almgren}, we find that for any $0<\rho<1$ 
$$
N^0(\rho,u_0^A)=\lim_{r_j\to 0}N^0(\rho,u_{r_j}^A)=\lim_{r_j\to 0}N^0(\rho r_j,u)=\ka,
$$
which implies that $u_0^A$ is parabolically homogeneous of degree $\ka$ in $S_1$ (see the proof of \cite{DanGarPetTo17}*{Theorem~7.3}), and by homogeneity, can be extended to $S_\infty$. In addition, by the Complementarity condition (Lemma~\ref{lem:par-compl-cond}), we have $u_0^A(0)=|\reallywidehat{\D u_0^A}(0)|=0$, where $\reallywidehat{\D u_0^A}$ is the even extension of $\D u_0^A$ from $S_1^+$ to $S_1$. Thus we can repeat the proof of \cite{DanGarPetTo17}*{Proposition~8.1} to get $\ka>1$. Then, it follows from \cite{DanGarPetTo17}*{Proposition~8.5} that either $\ka=3/2$ or $\ka\ge 2$.
\end{proof}

\begin{corollary}\label{cor:weiss-nonneg}
Let $u$ be an almost minimizer for the parabolic Signorini problem in $Q_1$ and $z_0\in \Gamma(u)\cap Q'_{1/2}$. Then for any $\ka_0>2$, $\e\in(0,\al]$ and $0<\de<2$, 
$$
W_{3/2,\al,\e,\de}(r,u,z_0)\ge 0\quad\text{for }0<r<r_0,
$$
where $r_0$ is as in Theorem~\ref{thm:par-weiss}.
\end{corollary}

\begin{proof}
The proof follows by using Lemma~\ref{lem:par-min-freq} and repeating the argument in \cite{JeoPet21}*{Corollary~6.3}.
\end{proof}


\section{Growth estimates}
In this section we establish the optimal growth of almost minimizers at free boundary points (Lemma~\ref{lem:par-opt-growth-est}).

Given $\ka\ge3/2$, we define the \emph{$\ka$-homogeneous rescalings} of $u$ at $z_0=(x_0,t_0)\in \Gamma(u)\cap Q'_{1/2}$ by 
$$
u_{z_0,r}(x,t):=u_{z_0,r}^{(\ka)}(x,t)=\frac{u(x_0+rx,t_0+r^2t)}{r^\ka},\quad(x,t)\in S_1.
$$

Note that $\widehat{N}_{\ka_0,\de}(0+,u,z_0)$ and $\N_\de(0+,u,z_0)$ are independent of $\al$ and $\e$.

\begin{lemma}[Weak growth estimates]\label{lem:par-weak-growth-est}
Suppose $u$ satisfies the almost parabolic Signorini property at $z_0\in\Ga(u)\cap Q'_{1/2}$. If $\widehat{N}_{\ka_0,\de}(0+,u,z_0)\ge \ka\ge 1$ for some $\ka\le\ka_0$, $\ka_0>2$ and $0<\de<2$, then for any $0<\we\le\al/2<1$, 
\begin{align}\label{eq:weak-growth-est}\begin{split}
    \int_{S_r(t_0)}u^2G_{z_0}\,dxdt &\le C(\ka_0,\we)\|u\|_{\fsz}^2r^{2\ka+2-\we},\\
    \int_{S_r(t_0)}2(t_0-t)|\D u|^2G_{z_0}\,dxdt &\le C(\ka_0,\we)\|u\|_{\fsz}^2r^{2\ka+2-\we}
\end{split}\end{align}
for $0<r<r_0=r_0(\ka_0,\we)$.
\end{lemma}

\begin{proof}
Without loss of generality we assume $z_0=0$. Note that for every $\e=2\we\in(0,\al]$, the condition $\widehat{N}_{\ka_0,\de}(0+,u)\ge\ka$ implies that $\widehat{N}_{\ka_0,\e,\de}(r,u)\ge \ka$ for $0<r<r_0(\ka_0,\e)$. Then we also have $\N_{\ka_0,\e,\de}(r,u)\ge \ka$ for such $r$, and thus 
\begin{align}\label{eq:Weiss-nonneg}
W_{\ka,\al,\e,\de}(r,u)=\frac{e^{ar^\al}}{r^{2\ka+2}}\left(\int_{S_r}u^2G\right)(1-br^\e)\left(\N_{\ka_0,\e,\de}(r,u)-\ka\right)\ge 0.
\end{align}
For $u_r=u_{0,r}^{(\ka)}$, we define
\begin{align}\label{eq:m}
m(r):=\int_{S_1}u_r^2G=\frac1{r^{2\ka+2}}\int_{S_r}u^2G.
\end{align}
Using
$$\frac d{dr}u_r(x,t)=-\frac1{r^{\ka+1}}\left(\ka u(rx,r^2t)-(rx)\cdot\D u(rx,r^2t)-2(r^2t)\pa_tu(rx,r^2t)\right),
$$
we can compute
\begin{align*}
    m'(r)&=2\int_{S_1}u_r(x,t)\left(\frac d{dr}u_r(x,t)\right)G(x,t)\,dxdt\\
    &\begin{multlined}=-\frac2{r^{2\ka+1}}\int_{S_1}u(rx,r^2t)\big(\ka u(rx,r^2t)-(rx)\cdot\D u(rx,r^2t)\\
    -2(r^2t)\pa_tu(rx,r^2t)\big)G(x,t)\,dxdt\end{multlined}\\
    &=-\frac2{r^{2\ka+3}}\int_{S_r}u(\ka u-x\cdot\D u-2t\pa_tu)G\,dxdt.
\end{align*}
By applying Theorem~\ref{thm:par-weiss}, we further have
\begin{align*}
    |m'(r)|=\frac2{r^{2\ka+3}}\left|\int_{S_r}u(\ka u-x\cdot\D u-2t\pa_tu)G\right|\le\frac{4r^{-\e/2}}\ka\frac d{dr}W_{\ka,\al,\e,\de}(r,u).
\end{align*}
This, along with \eqref{eq:Weiss-nonneg}, gives that for $0<s<r<r_0$ 
\begin{align*}
    |m(r)-m(s)|&\le\int_s^{r}|m'(\rho)|\,d\rho\le \frac4\ka\int_s^{r}\rho^{-\e/2}\frac d{d\rho}W_{\ka,\al,\e,\de}(\rho,u)\,d\rho\\
    &\le 4s^{-\e/2}\int_s^{r}\frac d{d\rho}W_{\ka,\al,\e,\de}(\rho,u)\,d\rho\le 4s^{-\e/2} W_{\ka,\al,\e,\de}(r,u).
\end{align*}
In particular, we have 
$$
m(r)\le m(r_0)+4r^{-\e/2} W_{\ka,\al,\e,\de}(r_0,u).
$$
This implies the first bound. The second one is then derived by utilizing the first one and the monotonicity $W_{\ka,\al,\e,\de}(r,u)\le W_{\ka,\al,\e,\de}(r_0,u)$.
\end{proof}

In the rest of this section, we remove the extra $\tilde\e>0$ in Lemma~\ref{lem:par-weak-growth-est} and obtain the optimal growth in the case of the least frequency $\ka=3/2$. To this end, we first derive the polynomial decay estimate of the Weiss-type energy $W_{3/2,\al,\e,\de}$. Following the approach in the elliptic counterpart \cite{JeoPet21}, we achieve this decay by utilizing the epiperimetric inequality. However, it is worth noting that employing the epiperimetric inequality in our context  is considerably more complex and technical.

Before we state the parabolic epiperimetric inequality from \cite{Shi20}, we introduce two types of ``standard'' Weiss energy functionals that will be used in this section.
\begin{align*}
    V^0_{3/2}(t,v)&:=\frac1{(-t)^{3/2}}\int_{\R^n}\left( (-2t)|\D v(x,t)|^2-\frac32 v(x,t)^2\right)G(x,t)\,dx,\\
    W_{3/2}^0(r,v)&:=\frac1{r^{5}}\int_{S_r}\left((-2t)|\D v|^2-\frac32 v^2\right)G\,dxdt.
\end{align*}

\begin{theorem}[Epiperimetric Inequality \cite{Shi20}]\label{thm:epi-ineq} Let $v$ be a solution of the parabolic Signorini problem in $S_1$. Then there is a dimensional constant $\xi\in(0,1)$ such that
$$
V^0_{3/2}(t/e,v)\le(1-\xi)V^0_{3/2}(t,v),\quad -1<t<0.
$$
\end{theorem}

\begin{lemma}\label{lem:par-weiss-bound}
    Fix $\ka_0>2$ and $0<\de<2$. Suppose $u$ satisfies the almost parabolic Signorini property at $z_0\in\Ga(u)\cap Q'_{1/2}$. Then there exists $\si=\si(n,\al)>0$ such that for any $\e\in(0,\al]$, 
    \begin{align}\label{eq:weiss-bound}
0\le W_{3/2,\al,\e,\de}(r,u,z_0)\le Cr^{\min\{\si,3\e/4\}},\quad 0<r<r_0(\ka_0,n,\al,\e)
\end{align}
with $C=C(\ka_0,n,\al,\e)\|u\|_{\fsz}^2$.
\end{lemma}

\begin{proof}
We split our proof into several steps.

\medskip\noindent\emph{Step 1.} We may assume $z_0=0$. We suppose that for $r\in(0,1)$ a.e., there exists a parabolic Signorini replacement $v$ of $u$ in $S_r$. For the simplicity of the presentation, we will derive \eqref{eq:weiss-bound} under this assumption in Step 1-Step 3 and consider the general case without the existence assumption in Step 4. 

By the epiperimetric inequality (Theorem~\ref{thm:epi-ineq}) and the monotonicity of $V^0_{3/2}$ (\cite{Shi20}*{Lemma~2.2}), we get \begin{align*}
    r^{5}W_{3/2}^0(r,v)&=\int_{-r^2}^0(-t)^{3/2} V_{3/2}^0(t,v)\,dt=\sum_{m=0}^\infty\int_{-\frac{r^2}{e^m}}^{-\frac{r^2}{e^{m+1}}}(-t)^{3/2} V_{3/2}^0(t,v)\,dt\\
    &\le \sum_{m=0}^\infty(1-\xi)^m\int_{-\frac{r^2}{e^m}}^{-\frac{r^2}{e^{m+1}}}(-t)^{3/2} V_{3/2}^0(e^mt,v)\,dt\\
    &=\sum_{m=0}^\infty\left(\frac{1-\xi}e\right)^m\int_{-r^2}^{-\frac{r^2}e}\left(-\frac s{e^m}\right)^{3/2} V_{3/2}^0(s,v)\,ds\\
    &\le \sum_{m=0}^\infty\left(\frac{1-\xi}{e^{5/2}}\right)^mV_{3/2}^0(-r^2,v)\int_{-r^2}^{-\frac{r^2}e}(-s)^{3/2}\,ds\\
    &=\frac{r^{5}(e^{5/2}-1)}{(5/2)(e^{5/2}+\xi-1)}V_{3/2}^0(-r^2,v).
\end{align*}
Thus \begin{align}\label{eq:par-epip-1}
    W_{3/2}^0(r,v)\le \frac{1-\eta}{5/2}V_{3/2}^0(-r^2,v)=\frac{1-\eta}{5/2}V_{3/2}^0(-r^2,u),
\end{align}
where $\eta:=\frac{\xi}{e^{5/2}+\xi-1}\in(0,1/e)$. On the other hand, by differentiating 
$$
 r^{5}W_{3/2}^0(r,u)=\int_{-r^2}^0(-t)^{3/2} V_{3/2}^0(t,u)\,dt
$$
with respect to $r$, we obtain after simplification \begin{align}\label{eq:par-weiss-diff-eq-W-V}
V_{3/2}^0(-r^2,u)=(5/2)W_{3/2}^0(r,u)+\frac r2\frac d{dr}W_{3/2}^0(r,u).
\end{align}
This, along with \eqref{eq:par-epip-1}, gives
\begin{align}\label{eq:par-weiss-min-alm-min}
W_{3/2}^0(r,v)\le(1-\eta)W_{3/2}^0(r,u)+\left(\frac{1-\eta}5\right)r\frac d{dr}W_{3/2}^0(r,u).
\end{align}
For $\we=\we(\al)\in(0,\al)$ to be determined later, by applying \eqref{eq:par-alm-min-def-kappa} with $\ka=3/2$ and Lemma~\ref{lem:par-weak-growth-est}, we have
\begin{align}\label{eq:par-weiss-alm-min-prop}\begin{split}
      W_{3/2}^0(r,u)&=\frac{r^\al}{r^{5}}\int_{S_r}(-2t)|\D u|^2+\frac1{r^{5}}\int_{S_r}\left((1-r^\al)(-2t)|\D u|^2-\frac32 u^2\right)G      \\
      &\le C(\ka_0,\al)\|u\|_{\fso}^2r^{\al-\we}+
      (1+r^{\al})W_{3/2}^0(r,v)\\
      &\qquad\begin{multlined}+\frac1{r^{5}}\int_{S_r}\big[ 3/2r^{\al} v^2+(3/2+2r^\al)(u-v)^2 \\
      -2\left(3/2u-x\cdot\D u-2t\pa_tu\right)(u-v)\big] G+\frac{2\|u\|_{\fso}^2}{r^{5}}e^{-\frac1r}.\end{multlined}
      \end{split}\end{align}
Combining this and \eqref{eq:par-weiss-min-alm-min}, we obtain \begin{align*}
    W_{3/2}^0(r,u)&\le (1+r^{\al})(1-\eta)W_{3/2}^0(r,u)+(1+r^{\al})\left(\frac{1-\eta}5\right)r\frac d{dr}W_{3/2}^0(r,u)\\
    &\qquad\begin{multlined}+\frac 1{r^{5}}\int_{S_r}\big[3/2r^{\al} v^2+(3/2+2r^\al)(u-v)^2\\
    -2(3/2 u-x\cdot\D u-2t\pa_tu)(u-v)\big]G+C(\ka_0,\al,u)r^{\al-\we},
\end{multlined}\end{align*}
which is equivalent to
\begin{align}\label{eq:par-weiss-diff-ineq-1}\begin{split}
    \frac d{dr}W_{3/2}^0(r,u)&\ge \frac{5(\eta-r^{\al}(1-\eta))}{(1+r^{\al})(1-\eta)} \frac{W_{3/2}^0(r,u)}r\\
    &\begin{multlined}[t]\qquad+\frac{5}{(1+r^{\al})(1-\eta)r^{6}}\int_{S_r}\big[ -3/2r^{\al} v^2-(3/2+2r^\al)(u-v)^2\\
   +2(3/2 u-x\cdot\D u-2t\pa_tu)(u-v)\big]G
\end{multlined}\\
    & \qquad-C(\ka_0,\al,u)r^{\al-\we-1}.\end{split}\end{align}


\medskip\noindent\emph{Step 2.} In this step, we simplify \eqref{eq:par-weiss-diff-ineq-1} by estimating the second term in its right-hand side. To this aim, we decompose
\begin{align*}
    \begin{split}&\frac1{r^{6}}\int_{S_r}\big[-3/2r^{\al} v^2-(3/2+2r^\al)(u-v)^2+2(3/2 u-x\cdot\D u-2t\pa_tu)(u-v)\big]G\\
    &\qquad=I+II+III.
\end{split}\end{align*}
Concerning $II$, we use Lemma~\ref{lem:par-sig-alm-min-diff-est-weight} and Lemma~\ref{lem:par-weak-growth-est} to get
\begin{align}\label{eq:par-weiss-bound-add-term-est-2}\begin{split}
II&= -\frac{3/2+2r^\al}{r^6}\int_{S_r}(u-v)^2G\ge -\frac{C}{r^6}\left(r^\al\int_{S_r}(-t)|\D u|^2+\|u\|_{\fso}^2e^{-\frac1r}\right)\\
&\ge -C(\ka_0,\al,u)r^{\al-\we-1}.
\end{split}\end{align}
Regarding $I$, from \eqref{eq:par-weiss-bound-add-term-est-2} and Lemma~\ref{lem:par-weak-growth-est} we infer
\begin{align}\label{eq:par-weiss-bound-add-term-est-1}\begin{split}
I&\ge-\frac{3r^\al}{r^6}\left(\int_{S_r}u^2G+\int_{S_r}(u-v)^2G\right)\ge -C(\ka_0,\al,u)r^{\al-\we-1}.
\end{split}\end{align}
It remains to consider $III$. Following the argument in the proof of Lemma~\ref{lem:par-(u-w)-transf}, we have for $\ka=3/2$ 
\begin{align*}
    &\int_{S_r}(\ka(u-v)-x\cdot\D(u-v)-2t\pa_t(u-v))(u-v)G\,dxdt\\
    &\qquad=\frac1{\pi^{n/2}}\int_{\R^n\times(-2\ln r,\infty)}\pa_\tau((\wu-\wv)^2)e^{-|y|^2}e^{-(\ka+1)\tau}\,dyd\tau\\
    &\qquad\ge \frac{\ka+1}{\pi^{n/2}}\int_{\R^n\times(-2\ln r,\infty)}(\wu-\wv)^2e^{-|y|^2}e^{-(\ka+1)\tau}\,dyd\tau\\
    &\qquad\ge 0.
\end{align*}
This, together with Young's inequality and \eqref{eq:par-weiss-bound-add-term-est-2}, yields
\begin{align}\begin{split}\label{eq:par-weiss-bound-add-term-est-3}
    III &=\frac2{r^{6}}\int_{S_r}(3/2(u-v)-x\cdot\D(u-v)-2t\pa_t(u-v))(u-v)G\\
    &\qquad\qquad\qquad\qquad+\frac2{r^{6}}\int_{S_r}(3/2 v-x\cdot\D v-2t\pa_tv)(u-v)G\\
    &\ge-\frac1{r^{6-\we}}\int_{S_r}(3/2 v-x\cdot\D v-2t\pa_tv)^2G-\frac1{r^{6+\we}}\int_{S_r}(u-v)^2G.
\end{split}\end{align}
The second term in the last line is estimated in \eqref{eq:par-weiss-bound-add-term-est-2}. To estimate the first one, we bring the following computation made in the proof of \cite{DanGarPetTo17}*{Theorem~13.1}
$$
\frac d{dr}W_{3/2}^0(r,v)\ge \frac2{r^{6}}\int_{S_r}(3/2 v-x\cdot\D v-2t\pa_tv)^2G.
$$
It then follows that
\begin{align*}
    &-\frac1{r^{6-\we}}\int_{S_r}(3/2 v-x\cdot\D v-2t\pa_tv)^2G\\
    &\ge -\frac{r^{\we}}2\frac d{dr}W^0_{3/2}(r,v)\\
    &=-r^{\we-1}V^0_{3/2}(-r^2,v)+5/2r^{\we-1}W^0_{3/2}(r,v)\\
    &\ge -r^{\we-1}V^0_{3/2}(-r^2,u)+\frac5{2(1+r^\al)}r^{\we-1}(W^0_{3/2}(r,u)-Cr^{\al-\we}+r(I+II+III))\\
    &\begin{multlined}\ge -Cr^{\al-\tilde\e-1}+O(r^{\we})\left(\frac{W_{3/2}^0(r,u)}r\right)+O(r^{\we})\left(\frac{d}{dr}W_{3/2}^0(r,u)\right)\\
    -r^{\we}\left(\frac1{r^{6-\we}}\int_{S_r}(3/2 v-x\cdot\D v-2t\pa_tv)^2G\right),\end{multlined}
\end{align*}
where we applied \eqref{eq:par-weiss-diff-eq-W-V} for $v$ in the third line, and used \eqref{eq:par-weiss-alm-min-prop} in the fourth line and \eqref{eq:par-weiss-diff-eq-W-V}, \eqref{eq:par-weiss-bound-add-term-est-2}, \eqref{eq:par-weiss-bound-add-term-est-1}, \eqref{eq:par-weiss-bound-add-term-est-3} in the last step. This implies
\begin{multline*}
    -\frac1{r^{6-\we}}\int_{S_r}(3/2 v-x\cdot\D v-2t\pa_tv)^2G\\
    \ge -Cr^{\al-\tilde\e-1}+O(r^{\we})\left(\frac{W_{3/2}^0(r,u)}r\right)+O(r^{\we})\left(\frac{d}{dr}W_{3/2}^0(r,u)\right),
\end{multline*}
which, combined with \eqref{eq:par-weiss-bound-add-term-est-3}, gives
\begin{align*}
III&\ge-C(\ka_0,\al,u)r^{\al-2\we-1}+O(r^{\we})\left(\frac{W_{3/2}^0(r,u)}r\right)+O(r^{\we})\left(\frac{d}{dr}W_{3/2}^0(r,u)\right).
\end{align*}
Now, by taking $\we=\al/3$, we conclude
\begin{multline*}
    I+II+III\\
    \ge-C(\ka_0,\al,u)r^{\al/3-1}+O(r^{\al/3})\left(\frac{W_{3/2}^0(r,u)}r\right)
    +O(r^{\al/3})\left(\frac d{dr}W_{3/2}^0(r,u)\right).
\end{multline*}
Therefore, \eqref{eq:par-weiss-diff-ineq-1} can be simplified to
\begin{multline}
    \label{eq:par-weiss-diff-ineq-2}
        \frac d{dr}W_{3/2}^0(r,u)\ge \left(\frac{5\eta}{1-\eta}+O(r^{\al/3})\right)\frac{W_{3/2}^0(r,u)}r-C(\ka_0,\al,u)r^{\al/3-1}.
\end{multline}

\medskip\noindent\emph{Step 3.} We consider the Weiss-type energy $W_{3/2,\al}=W_{3/2,\al,\al,1}$ with $\e=\al$ and $\de=1$. By Corollary~\ref{cor:weiss-nonneg} and Lemma~\ref{lem:par-weak-growth-est}, 
\begin{align}
    \label{eq:par-weiss-lower-bound}\begin{split}
        W_{3/2}^0(r,u)&=e^{-ar^\al}W_{3/2,\al}(r,u)-\frac{3/2 br^{\al}}{r^{5}}\int_{S_r}u^2G-\frac{\|u\|_{\fso}^2}{r^{5}}e^{-\frac1{r}}r^{-1}\\
        &\ge-C_0(\ka_0,\al,u)r^{\al/2}, \quad 0<r<r_0(\ka_0,\al).
    \end{split}
\end{align}
We recall $\eta\in(0,1/e)$ and use \eqref{eq:par-weiss-diff-ineq-2} and \eqref{eq:par-weiss-lower-bound} to get the differential inequality for $W^0_{3/2}(r,u)$: \begin{align*}
    \frac{d}{dr}W_{3/2}^0(r,u)&\ge\left(\frac{5\eta}{1-\eta}+O(r^{\al/3})\right)\left(\frac{W_{3/2}^0(r,u)+C_0r^{\al/3}}r-\frac{C_0r^{\al/3}}r\right)-Cr^{\al/3-1}\\
    &\ge 5\eta\left(\frac{W_{3/2}^0(r,u)+C_0r^{\al/3}}r\right)-Cr^{\al/3-1}\\
    &\ge 5\eta\frac{W_{3/2}^0(r,u)}r-C_1r^{\al/3-1}.
\end{align*}
We take $\si=\si(n,\al)$ such that $0<\si<\min\{5\eta,\al/3\}$, and use the differential inequality and \eqref{eq:par-weiss-lower-bound} to obtain \begin{align*}
    &\frac d{dr}\left[W_{3/2}^0(r,u)r^{-\si}+\frac{2C_1}{\al/3-\si}r^{\al/3-\si}\right]\\
    &\qquad=r^{-\si}\left(\frac d{dr}W_{3/2}^0(r,u)-\frac\si{r}W_{3/2}^0(r,u)\right)+2C_1r^{\al/3-\si-1}\\
    &\qquad\ge r^{-\si-1}(5\eta-\si)W_{3/2}^0(r,u)+C_1r^{\al/3-\si-1}\\
    &\qquad\ge -C_2(\ka_0,\al)\|u\|_{\fso}^2r^{\al/2-\si-1}+C_3(\ka_0,\al)\|u\|_{\fso}^2r^{\al/3-\si-1}\\
    &\qquad\ge 0,\quad 0<r<r_0(\ka_0,n,\al).
\end{align*}
This readily gives
$$
W_{3/2}^0(r,u)\le C(\ka_0,n,\al,u)r^\si.
$$
To complete the proof, let $\e\in(0,\al]$ and $\de\in(0,2)$ be given. Then, by applying Lemma~\ref{lem:par-weak-growth-est} (with $\we=\e/4$), we conclude that  
\begin{align*}
    W_{3/2,\al,\e,\de}(r,u)&=e^{ar^\al}W_{3/2}^0(r,u)+\frac{3/2 e^{ar^\al}br^\e}{r^{5}}\int_{S_r}u^2G+\frac{\|u\|_{\fso}^2e^{ar^\al}}{r^{5}}e^{-\frac1{r}}r^{-\de}\\
    &\le C(\ka_0,n,\al,\e)\|u\|_{\fso}^2r^{\min\{\si,3\e/4\}},\quad 0<r<r_0(\ka_0,n,\al,\e).
\end{align*}

\medskip\noindent\emph{Step 4}. To close the argument, we need to remove the assumption on the existence of the parabolic Signorini replacement made in Step 1. To this end, we consider $u_\mu=u*\vp_\mu$ as in \eqref{eq:conv}. Then, for $r\in(0,1)$ a.e., the parabolic Signorini replacement $v_\mu$ of $u_\mu$ in $S_r$ exists. We observe that only the following properties of $u$ are used in Step 1 and Step 2: the almost parabolic Signorini property (equation \eqref{eq:par-alm-min-def}), Lemma~\ref{lem:par-sig-alm-min-diff-est-weight} and the weak growth estimates with $\ka=3/2$ (Lemma~\ref{lem:par-weak-growth-est}). We have already seen in Remark~\ref{rmk:conv-alm-min} and Corollary~\ref{cor:conv-diff-est} that $u_\mu$ satisfies analogues of the first two properties. Moreover, by using the triangle inequality, it is easily seen that $u_\mu$ satisfies the following analogue of \eqref{eq:weak-growth-est} with $\ka=3/2$: for any $0<\tilde\e\le\al/2$
\begin{align*}
    \int_{S_r}u_\mu^2G&\le C(\ka_0,\al)\|u_\mu\|_{\fs_0}^2r^{5-\tilde\e}+C(r,\al)\|u-u_\mu\|_{\fs_0}^2,\\
    \int_{S_r}(-2t)|\D u_\mu|^2G&\le C(\ka_0,\al)\|u_\mu\|_{\fs_0}^2r^{5-\tilde\e}+C(r,\al)\|u-u_\mu\|_{\fs_0}^2.
\end{align*}

Now, with these properties of $u_\mu$ at hand, we can follow the argument in Step 1 and Step 2 with $u_{\mu}$ and $v_{\mu}$ in the place of $u$ and $v$ to deduce an analogue of \eqref{eq:par-weiss-diff-ineq-2}:
\begin{align*}
        \frac d{dr}W_{3/2}^0(r,u_\mu)&\ge \left(\frac{5\eta}{1-\eta}+O(r^{\al/6})\right)\frac{W_{3/2}^0(r,u_\mu)}r-C(\ka_0,\al)\|u_\mu\|_{\fso}^2r^{\al/6-1}\\
        &\qquad-C(r,\al,\ka_0)\|u-u_\mu\|^2_{\fs_0}.
\end{align*}
Taking $\mu\to 0$, we obtain the differential inequality \eqref{eq:par-weiss-diff-ineq-2} concerning $W^0_{3/2}(r,u)$ for $r\in(0,1)$ a.e., but with $\al/6$ in the place of $\al/3$. Then, \eqref{eq:weiss-bound} readily follows by arguing as in Step 3 with obvious modifications.
\end{proof}

As in \cite{JeoPet21}, by using the polynomial decay of $W_{3/2,\al,\e,\de}$ we can improve Lemma~\ref{lem:par-weak-growth-est} when $\ka=3/2$ and derive the optimal growth.

\begin{lemma}[Optimal growth estimate]\label{lem:par-opt-growth-est}
Fix $\ka_0>2$. Suppose that $u\in \fsz$ satisfies the almost parabolic Signorini property at $z_0\in\Ga(u)\cap Q'_{1/2}$. Then, 
\begin{align*}
    \int_{S_r(t_0)}u^2G_{z_0}\,dxdt &\le C(\ka_0,n,\al)\|u\|_{\fsz}^2r^{5},\\
    \int_{S_r(t_0)}2(t_0-t)|\D u|^2G_{z_0}\,dxdt &\le C(\ka_0,n,\al)\|u\|_{\fsz}^2r^{5}
\end{align*}
for $0<r<r_0=r_0(\ka_0,n,\al)$.
\end{lemma}

\begin{proof}
We may assume $z_0=0$. Take $\e=\e(n,\al)>0$ small so that $3\e/4<\si$ for $\si=\si(n,\al)$ as in Lemma~\ref{lem:par-weiss-bound}. Following the computation in the proof of Lemma~\ref{lem:par-weak-growth-est} with $\de=1$, we see that for any $0<s<r<r_0(\ka_0,n,\al)$, $$
|m(r)-m(s)|\le 4s^{-\e/2}W_{3/2,\al,\e,1}(r).
$$
By Lemma~\ref{lem:par-weiss-bound}, we further have
\begin{align*}
|m(r)-m(s)|\le Cs^{-\e/2}r^{3\e/4},
\end{align*}
with $C=C(\ka_0,n,\al)\|u\|_{\fso}^2$. Then, by a dyadic argument, we can obtain that 
\begin{align}\label{eq:m-est}
|m(r)-m(s)|\le Cr^{\e/4}.
\end{align}
Indeed, let $k=0,1,2,\cdots,$ be such that $r/2^{k+1}<s\le r/2^k$. Then, \begin{align*}
    |m(r)-m(s)|&\le\sum_{j=1}^k|m(r/2^{j-1})-m(r/2^{j})|+|m(r/2^k)-m(s)|\\
    &\le C\sum_{j=1}^{k+1}(r/2^j)^{-\e/2}(r/2^{j-1})^{3\e/4}=C\left(r^{1/4}2^{3/4}\right)^\e\sum_{j=1}^{k+1}2^{-j/4}\\
    &\le Cr^{\e/4}.
\end{align*}
In particular, we have $$
m(r)\le m(r_0)+Cr_0^{\e/4}\le C(\ka_0,n,\al)\|u\|_{\fso}^2,\quad 0<r<r_0(\ka_0,n,\al).
$$
This implies the first bound. The second bound follows from the first one and the monotonicity $W_{3/2,\al,\e,1}(r,u)\le W_{3/2,\al,\e,1}(r_0,u)$.
\end{proof}


\section{$3/2$-Homogeneous blowups}\label{sec:par-32-homog-blow}
In this section, we consider the so-called $3/2$-homogeneous blowups of almost minimizers at free boundary points. They are the limits of $3/2$-homogeneous rescalings, which are well-defined thanks to the optimal growth estimates. We achieve their uniqueness through controlling the ``rotation'' of the rescalings.

Concerning the $\ka$-homogeneous rescalings, for the rest of this paper, we focus exclusively on the case $\ka=3/2$. Thus we simply write $u_{z_0,r}=u_{z_0,r}^{(3/2)}$.

Fix $z_0=(x_0,t_0)\in \Gamma(u)\cap Q'_{1/2}$ and $R>1$, and let $r_0=r_0(\ka_0,n,\al)$ be as stated in Lemma~\ref{lem:par-opt-growth-est}. We have for $0<r<r_0$
\begin{align*}
    \int_{S_R}(-2t)|\D u_{z_0,r}|^2G&=\frac1{r^5}\int_{S_{Rr}(t_0)}2(t_0-t)|\D u|^2G_{z_0}\le C(\ka_0,n,\al)\|u\|_{\fsz}^2R^5,\\
    \int_{S_R}u_{z_0,r}^2G&=\frac1{r^5}\int_{S_{Rr}(t_0)}u^2G_{z_0}\le C(\ka_0,n,\al)\|u\|_{\fsz}^2R^5.
\end{align*}
Thus, for a sequence $r=r_j\to 0+$, $u_{z_0,r_j}\to u_{z_0,0}$ weakly in $W^{1,0}_{2,\loc}(S_R,G)$. Moreover, $u_{z_0,r}$ is an unweighted almost minimizer with a gauge function $\eta_r(\rho)=(r\rho)^\al\le \rho^\al$. Given $\e>0$ and $K\Subset Q_R\cap\{t\le-\e\}$, we infer from Theorem~\ref{thm:par-alm-min-grad-holder} that there is a constant $C>0$ such that for any $0<r<r_0$, 
\begin{align*}
    \|u_{z_0,r}\|_{C^{\al,\al/2}(K)}+\|\D u_{z_0,r}\|_{C^{\be,\be/2}(K^\pm\cup K')}&\le C\|u_{z_0,r}\|_{W^{1,0}_2(Q_R\cap\{t\le-\e\})}\\
    &\le C\|u_{z_0,r}\|_{W^{1,0}_2(S_R,G)},
\end{align*}
and hence over a sequence $r=r_j\to 0+$ \begin{align*}
    u_{z_0,r_j}\to u_{z_0,0}\quad\text{in }C^{1,0}_{\loc}((Q_R^\pm\cup Q'_R)\cap\{t\le-\e\}).
\end{align*}
Now, taking $\e\to 0$ and $R\to \infty$ and using Cantor's diagonal argument, we can find a subsequence $r=r_j\to 0+$ such that for some $u_{z_0,0}\in C^{1,0}_{\loc}(S^\pm_\infty\cup S'_\infty)$ $$
u_{z_0,r_j}\to u_{z_0,0}\quad\text{in }C^{1,0}_{\loc}(S^\pm_\infty\cup S'_{\infty}).
$$
We call such $u_{z_0,0}$ a \emph{$3/2$-homogeneous blowup} of $u$ at $z_0$.

\begin{lemma}[Rotation estimate]\label{lem:par-rot-est}
Suppose that $u$ satisfies the almost parabolic Signorini property at $z_0\in\Ga(u)\cap Q_{1/2}'$. Then there exists $\si=\si(n,\al)>0$ such that for any $0<s<r<r_0=r_0(\ka_0,n,\al)$ and $-1<t<0$,
\begin{align*}
    \int_{\R^n}|u_{z_0,r}(x,t)-u_{z_0,s}(x,t)|G_{z_0}(x,t)\,dx\le C(-t)^{3/4+\si}r^{2\si}
    \end{align*}
with $C=C(\ka_0,n,\al)\|u\|_{\fsz}^2$.
\end{lemma}

\begin{proof}
Without loss of generality, we may assume $z_0=0$. We fix $\e=\al$, $\de=1$ and $\ka=3/2$, and simply write $W_{\ka,\rho}=W_{\ka,\al,\e,\de,\rho}$. By using \eqref{eq:par-weiss-deriv-1} in Theorem~\ref{thm:par-weiss}, we infer that for $R>R_0(\ka_0,\al)$
\begin{multline*}
 W_{\ka,e^{-R/2}}(e^{3-R/2},u)-W_{\ka,e^{-R/2}}(e^{1-R/2},u)\\
 \ge \frac{2\ka+1}{\pi^{n/2}}\int_{R-6}^{R-2}\frac{e^{-(\ka+1)(r+R)}}{(e^{-(\ka+1)r}-e^{-(\ka+1)R})^2}\int_{\R^n}(\wu(y,R)-\wu(y,r))^2e^{-|y|^2}\,dydr.
\end{multline*}
Since $\ka=3/2$, we have for $R-6<r<R-2$
\begin{align*}
    \frac{e^{-(\ka+1)(r+R)}}{(e^{-(\ka+1)r}-e^{-(\ka+1)R})^2}\ge \frac{e^{-(\ka+1)(2R-2)}}{(e^{-(\ka+1)(R-6)}-e^{-(\ka+1)R})^2}=\frac{e^5}{(e^{15}-1)^2},
\end{align*}
thus
\begin{multline*}
    \int_{R-6}^{R-2}\int_{\R^n}(\wu(y,R)-\wu(y,r))^2e^{-|y|^2}\,dydr\\
    \le C(n)\left(W_{\ka,e^{-R/2}}(e^{3-R/2},u)-W_{\ka,e^{-R/2}}(e^{1-R/2},u)\right).
\end{multline*}
To estimate the right-hand side of this previous inequality, we note that by Lemma~\ref{lem:par-weiss-bound}, there is $\si=\si(n,\al)>0$ such that for $W_\ka=W_{\ka,0}$ (i.e., $W_{\ka,\al,\e,\de,\rho}$ with $\ka=3/2$, $\e=\al$, $\de=1$ and $\rho=0$) and for $R>R_0=R_0(\ka_0,n,\al)$, 
\begin{align*}
    0\le W_\ka(e^{-R/2}, u)\le Ce^{-2\si R},\quad W_\ka(e^{1-R/2}, u)\ge 0,\quad W_\ka(e^{3-R/2},u)\le Ce^{-2\si R}
\end{align*}
with $C=C(\ka_0,n,\al)\|u\|_{\fso}^2$. Then 
\begin{align*}
    &W_{\ka,e^{-R/2}}(e^{3-R/2},u)\\
    &\qquad\begin{multlined}
     =\frac{e^{ae^{(3-R/2)\al}}}{e^{-(\ka+1)(R-6)}-e^{-(\ka+1)R}}\bigg(\int_{S_{e^{3-R/2}}}((-2t)|\D u|^2-\ka(1-be^{(3-R/2)\al})u^2)G\\
     -\int_{S_{e^{-R/2}}}((-2t)|\D u|^2-\ka(1-be^{(3-R/2)\al})u^2)G+\|u\|_{\fso}^2e^{-e^{R/2-3}}e^{R/2-3}\bigg)
    \end{multlined}\\
    &\qquad\begin{multlined}
        =\frac1{e^{6(\ka+1)}-1}\bigg(e^{6(\ka+1)}W_\ka(e^{3-R/2},u)\\
        \qquad-\frac{e^{ae^{(3-R/2)\al}}}{e^{-(\ka+1)R}}\int_{S_{e^{-R/2}}}((-2t)|\D u|^2-\ka(1-be^{(3-R/2)\al})u^2)G\bigg)
    \end{multlined}\\
    &\qquad\le \frac1{e^{6(\ka+1)}-1}\left(e^{6(\ka+1)}W_\ka(e^{3-R/2},u)-W_\ka(e^{-R/2},u)+O(e^{-\al/2R})\|u\|_{\fso}^2\right)\\
    &\qquad\le Ce^{-2\si R}.
\end{align*}
Similarly, \begin{align*}
    &W_{\ka,e^{-R/2}}(e^{1-R/2},u)\\
    &\qquad\begin{multlined}
        =\frac1{e^{2(\ka+1)}-1}\bigg(e^{2(\ka+1)}W_\ka(e^{1-R/2},u)\\
        \qquad-\frac{e^{ae^{(1-R/2)\al}}}{e^{-(\ka+1)R}}\int_{S_{e^{-R/2}}}((-2t)|\D u|^2-\ka(1-be^{(1-R/2)\al})u^2)G\bigg)
    \end{multlined}\\
    &\qquad\ge \frac1{e^{2(\ka+1)}-1}\left(e^{2(\ka+1)}W_\ka(e^{1-R/2},u)-W_\ka(e^{-R/2},u)+O(e^{-\al/2R})\|u\|_{\fso}^2\right)\\
    &\qquad\ge -Ce^{-2\si R}.
\end{align*}
Thus \begin{align*}
    \int_{R-6}^{R-2}\int_{\R^n}(\wu(y,R)-\wu(y,r))^2e^{-|y|^2}\,dydr\le Ce^{-2\si R},
\end{align*}
and hence by Cauchy-Schwarz inequality
\begin{align*}
    \int_{R-6}^{R-2}\int_{\R^n}|\wu(y,R)-\wu(y,r)|e^{-|y|^2}\,dydr\le C_0e^{-\si R},\quad R>R_1(\ka_0,n,\al).
\end{align*}
Then, for $R>R_1$ and $3<\eta<6$, \begin{align}
    \label{eq:par-rot-est-1}
    \begin{split}
       &\int_{\R^n}\left|\wu(y,R+\eta)-\int_R^{R+1}\wu(y,\tau)\,d\tau\right|e^{-|y|^2}\,dy\\
        &\qquad\le \int_{(R+\eta)-6}^{(R+\eta)-2}\int_{\R^n}\left|\wu(y,R+\eta)-\wu(y,\tau)\,d\tau\right|e^{-|y|^2}\,dyd\tau\\
        &\qquad\le C_0e^{-\si(R+\eta)}.
    \end{split}
\end{align}
We claim that for any $k\in \mathbb N$, $3<\eta<5$ and $R>R_1$, 
\begin{align}
    \label{eq:par-rot-est-iter}\begin{split}
    \int_{\R^n}\left|\wu(y,R+k\eta)-\int_R^{R+1}\wu(y,\tau)\,d\tau\right|e^{-|y|^2}\,dy&\le C_0\sum_{j=1}^ke^{-\si(R+\eta j)}.
\end{split}\end{align}
Indeed, we prove it by induction on $k\in\mathbb{N}$. \eqref{eq:par-rot-est-iter} is true for $k=1$ by \eqref{eq:par-rot-est-1}. If \eqref{eq:par-rot-est-iter} is true for $k-1$, then the induction hypothesis and \eqref{eq:par-rot-est-1} yield
\begin{align*}
    &\int_{\R^n}\left|\wu(y,R+k\eta)-\int_R^{R+1}\wu(y,\tau)\,d\tau\right|e^{-|y|^2}\,dy\\
    &\qquad\le\int_{\R^n}\left|\wu(y,R+\eta+(k-1)\eta)-\int_{R+\eta}^{R+\eta+1}\wu(y,\rho)\,d\rho\right|e^{-|y|^2}\,dy\\
    &\qquad\qquad+\int_{\R^n}\left|\int_{R+\eta}^{R+\eta+1}\wu(y,\rho)\,d\rho-\int_R^{R+1}\wu(y,\tau)\,d\tau\right|e^{-|y|^2}\,dy\\
    &\qquad\le C_0\sum_{j=1}^{k-1}e^{-\si(R+\eta+\eta j)}+\int_{R+\eta}^{R+\eta+1}\int_{\R^n}\left|\wu(y,\rho)-\int_R^{R+1}\wu(y,\tau)\,d\tau\right|e^{-|y|^2}\,dyd\rho\\
    &\qquad\le C_0\sum_{j=2}^ke^{-\si(R+\eta j)}+\int_{R+\eta}^{R+\eta+1}C_0e^{-\si\rho}\,d\rho\le C_0\sum_{j=1}^ke^{-\si(R+\eta j)}.
\end{align*}
Now, let $S>R>R_1(\ka_0,n,\al)+12$ be given. Then we can choose $k\in \mathbb{N}$ and $\eta\in(3,5)$ such that $S=R-12+\eta k$. By \eqref{eq:par-rot-est-iter}, we have 
\begin{align*}
    &\int_{\R^n}|\wu(y,R)-\wu(y,S)|e^{-|y|^2}\,dy\\
    &\qquad\le \int_{\R^n}\left|\wu(y,R-12+3\cdot 4)-\int_{R-12}^{R-11}\wu(y,\tau)\,d\tau\right|e^{-|y|^2}\,dy\\
    &\qquad\qquad+\int_{\R^n}\left|\int_{R-12}^{R-11}\wu(y,\tau)\,d\tau-\wu(y,R-12+k\eta)\right|e^{-|y|^2}\,dy\\
    &\qquad \le C_0\sum_{j=1}^{3}e^{-\sigma(R-12+4j)}+C_0\sum_{j=1}^{k}e^{-\sigma(R-12+\eta j)}\le C(\ka_0,n,\al)\|u\|_{\fso}^2e^{-\si R}.
    \end{align*}
To complete the proof, define $r_0=r_0(\ka_0,n,\al)=e^{-1/2(R_1(\ka_0,n,\al)+12)}$ and let $-1<t<0$ and $0<s<r<r_0$ be given. By using $u_r(x,t)=\frac{u(rx,r^2t)}{r^\ka}=(-t)^{\ka/2}\wu\left(\frac{x}{2\sqrt{-t}},-\ln(-t)-2\ln r\right)$, we conclude
\begin{align*}
    &\int_{\R^n}|u_r(x,t)-u_s(x,t)|G(x,t)\,dx\\
    &\qquad=\frac{(-t)^{\ka/2}}{\pi^{n/2}}\int_{\R^n}|\wu(y,-\ln(-t)-2\ln r)-\wu(y,-\ln(-t)-2\ln s)|e^{-|y|^2}\,dy\\
    &\qquad\le C(\ka_0,n,\al)\|u\|_{\fso}^2(-t)^{3/4+\si}r^{2\si}.\qedhere
    \end{align*}
\end{proof}

\begin{lemma}\label{lem:par-blowup-rot-est}
Let $u,z_0,\sigma,r_0,C$ be as in Lemma~\ref{lem:par-rot-est}. Then, for $0<r<r_0$ and $-1<t<0$, \begin{align*}
    \int_{\R^n}|u_{z_0,r}(x,t)-u_{z_0,0}(x)|G(x,t)\,dx\le C(-t)^{3/4+\si}r^{2\si}.
\end{align*}
In particular, the blowup $u_{z_0,0}$ is unique.
\end{lemma}

\begin{proof}
If $u_{z_0,0}$ is the limit of $u_{z_0,s_j}$, $s_j\to 0$, then the first part of the lemma follows from Lemma~\ref{lem:par-rot-est} by taking $s_j\to 0$. For the second part, let $\overline{u}_{z_0,0}$ be another blowup. Then we have 
$$
\int_{\R^n}|u_{z_0,0}(x,t)-\overline{u}_{z_0,0}(x,t)|G(x,t)\,dx=0,\quad-1<t<0,
$$
thus $u_{z_0,0}=\overline{u}_{z_0,0}$.
\end{proof}


\section{Regularity of the regular set}
In this last section, we prove one of the most crucial results in this paper, the regularity of the regular set.

Recall that the limit $\hat N_{\ka_0,\delta}(0+,u,z_0)=\lim_{r\to0+}\hat N_{\ka_0,\e,\de}(r,u,z_0)$ is independent of $\e$.

\begin{definition}[Regular points]\label{def:par-reg-pt}
Let $u$ be an almost minimizer for the parabolic Signorini problem in $Q_1$. We say that a free boundary point $z_0\in Q'_{1/2}$ is \emph{regular} if $$
\widehat{N}_{\ka_0,\de}(0+,u,z_0)=3/2\quad\text{for some }\ka_0>2\,\,\,\text{and }0<\de<2.
$$
We denote the set of all regular points of $u$ by $\mathcal{R}(u)$ and call it \emph{regular set}.
\end{definition}

In view of Lemma~\ref{lem:par-Almgren}, we have at every regular point $z_0$
$$
N^0(0+,u,z_0)=3/2.
$$
In addition, regular points have the following characterization.

\begin{remark}
$z_0$ is a regular point if and only if $$
\lim_{\de\to 0}N_{\de}(0+,u,z_0)=\inf_{0<\de<2}N_{\de}(0+,u,z_0)=3/2.
$$
\end{remark}

\begin{proof}
If $z_0$ is a regular point, then $3/2=\widehat N_{\ka_0,\delta_0}(0+,u,z_0)=N_{\de_0}(0+,u,z_0)$ for some $\ka_0>2$ and $0<\delta_0<2$. This, along with Lemma~\ref{lem:par-min-freq} and the fact that $\delta\mapsto N_\de(0+,u,z_0)$ is nondecreasing, implies $3/2\le \widehat N_{\ka_0,\de}(0+,u,z_0)\le \widehat N_{\ka_0,\de_0}(0+,u,z_0)=3/2$ for every $\de\in(0,\de_0)$, which readily gives $3/2=N_\de(0+,u,z_0)$ for any $0<\de<\de_0$. Therefore, we get $\lim_{\de\to0}N_\de(0+,u,z_0)=\inf_{0<\de<2}N_\de(0+,u,z_0)=3/2$.\\
To prove the opposite direction, we fix $\ka_0>2$. Take $\de_1>0$ such that $N_\de(0+,u,z_0)<7/4$ for $\de\in(0,\de_1)$. Then, by Lemmas~\ref{lem:par-Almgren} and ~\ref{lem:par-min-freq}, $3/2\le\widehat{N}_{\ka_0,\de}(0+,u,z_0)=N^0(0+,u,z_0)\le N_\de(0+,u,z_0)$ for $0<\de<\de_1$. Taking $\de\to 0$ yields $N^0(0+,u,z_0)=3/2$. This in turn gives that $\widehat{N}_{\ka_0,\de}(0+,u,z_0)=3/2$ for $0<\de<\de_1$, and we conclude that $z_0$ is a regular point.
\end{proof}

With the monotonicity of the frequency $\hat N_{\ka_0,\e,\de}$ (Theorem~\ref{thm:par-almgren}) and the frequency gap (Lemma~\ref{lem:par-min-freq}) at hand, we can prove the relative openness of the regular set by following the argument in \cite{JeoPet21}*{Corollary~9.5}.

\begin{corollary}\label{cor:par-reg-set-open}
The regular set $\mathcal{R}(u)$ is a relatively open subset of $\Ga(u)$.
\end{corollary}

\begin{lemma}[Nondegeneracy at regular points]\label{lem:par-nondeg-reg-pt}
Suppose that $u$ satisfies the almost parabolic Signorini property at $z_0=(x_0,t_0)\in\mathcal{R}(u)$. Then $$
\liminf_{t\to 0}\int_{S_1}\left(u_{z_0,t}\right)^2G=\liminf_{t\to 0}\frac1{t^5}\int_{S_r(t_0)}u^2G_{z_0}>0.
$$
\end{lemma}

\begin{proof}
By using \eqref{eq:m-est} and the Weiss-type monotonicity formula, we can employ the contradiction argument as in \cite{JeoPet21}*{Lemma~9.2} to prove Lemma~\ref{lem:par-nondeg-reg-pt}.
\end{proof}

\begin{proposition}\label{prop:par-gap-32-hom-classif}
If $u$ satisfies the almost parabolic Signorini property at $z_0\in\mathcal{R}(u)$, then $$
u_{z_0,0}(x,t)=c_{z_0}\Re(x'\cdot e_{z_0}+i|x_n|)^{3/2}\quad\text{in }S_\infty
$$
for some $c_{z_0}>0$ and $e_{z_0}\in \pa B'_1$.
\end{proposition}

\begin{proof}
Without loss of generality, we may assume $z_0=0$. Let $r_j\to 0+$ be a sequence such that $u_{r_j}\to u_0$ in $C^{1,0}_{\loc}(S^\pm_\infty\cup S'_\infty)$. Fix $R>1$, and consider $j$ large so that $Rr_j<1$. For such $r_j$, we take $\mu_j>0$ small so that $u_{\mu_j}=u*\vp_{\mu_j}$ as in \eqref{eq:conv} satisfies $C_0(r_j,\al)\|u-u_{\mu_j}\|_{\fs_0}^2\to0$ as $r_j\to0$, where $C_0(r_j,\al)$ is as in Corollary~\ref{cor:conv-diff-est}. We let $v_{\mu_j}$ be the parabolic Signorini replacement of $u_{\mu_j}$ in $S_R$, and denote its $3/2$-homogeneous rescaling by $(v_{\mu_j})_{r_j}(x,t)=\frac{v_{\mu_j}(r_jx,r_j^2t)}{r_j^{3/2}}$. Then, by Corollary~\ref{cor:conv-diff-est} and Lemma~\ref{lem:par-opt-growth-est}, we have
\begin{align*}
    &\int_{S_R}(-t)|\D(u_{r_j}-(v_{\mu_j})_{r_j}|^2G=\frac1{r_j^5}\int_{S_{Rr_j}}(-t)|\D(u-v_{\mu_j})|^2G\\
    &\le \frac2{r_j^5}\left(\int_{S_{Rr_j}}(-t)|\D(u_{\mu_j}-v_{\mu_j})|^2G+\|u-u_{\mu_j}\|_{\fs_0}^2\right)\\
    &\le\frac{C}{r_j^5}\left((Rr_j)^{\al/2}\int_{S_{Rr_j}}(-t)|\D u_{\mu_j}|^2G+\|u_{\mu_j}\|_{\fs_0}^2e^{-\frac1{r_j}}+C_0(r_j,\al)\|u-u_{\mu_j}\|_{\fs_0}^2\right)\\
    &\le\frac{C}{r_j^5}\left((Rr_j)^{\al/2}\int_{S_{Rr_j}}(-t)|\D u|^2G+\|u\|_{\fs_0}^2e^{-\frac1{r_j}}+C_0(r_j,\al)\|u-u_{\mu_j}\|_{\fs_0}^2\right)\\
    &\to0\quad\text{as }r_j\to0.
\end{align*}
Similarly, we can obtain
$$
\int_{S_R}(u_{r_j}-(v_{\mu_j})_{r_j})^2G\to0\quad\text{as }r_j\to0.
$$
Thus $(v_{\mu_j})_{r_j}\to u_0$ weakly in $W^{1,0}_{2,\loc}(S_R,G)$, and hence $u_0$ is a solution of the parabolic Signorini problem in $S_R$. Since $R>1$ is arbitrary, we see that $u_0$ is the solution in $S_\infty$. 

Next, we compare $u_r$ and Almgren rescalings $u_r^A$
$$
u_r=u_r^A\la(r),\quad\la(r)=\frac{\left(\frac1{r^2}\int_{S_r}u^2G\right)^{1/2}}{r^{3/2}}.
$$
It follows from Lemma~\ref{lem:par-opt-growth-est} and Lemma~\ref{lem:par-nondeg-reg-pt} that
$$
0<\liminf_{r\to 0+}\la(r)\le \limsup_{r\to 0+}\la(r)<\infty.
$$
Thus, for a sequence $r_j\to 0+$, $u_0=\la_0u_0^A$ for some constant $\mu_0\in(0,\infty)$. We have shown in  the proof of Lemma~\ref{lem:par-min-freq} that $u_0^A$ is $3/2$-parabolically homogeneous in $S_1$. Therefore, $u_0$ is also $3/2$-parabolically homogeneous in $S_1$, which can be extended to $S_\infty$ by applying the unique continuation for caloric function in $S_\infty^{\pm}$. In view of \cite{DanGarPetTo17}*{Proposition~8.5}, we conclude that 
\begin{equation*}
u_0(x,t)=c\Re(x'\cdot e+i|x_n|)^{3/2}\quad\text{in }S_\infty,\quad c>0,\,e\in\pa B'_1.\qedhere
\end{equation*}
\end{proof}

\begin{lemma}[Continuous dependence of blowups]\label{lem:par-blowup-est}
Let $u\in \fs$ be an almost minimizer for the parabolic Signorini problem in $Q_1$. If $z_1$, $z_2\in\mathcal{R}(u)$ and $\|z_1-z_2\|<r_1$, then 
$$
\int_{\pa B_1}|u_{z_1,0}-u_{z_2,0}|\,dSx\le C\|z_1-z_2\|^\g,
$$
with $r_1=r_1(\ka_0,n,\al)$, $C=C(\ka_0,n,\al,u)$ and $\g=\g(n,\al)>0$.
\end{lemma}

\begin{proof}
Let $r_0=r_0(\ka_0,n,\al)$ and $\sigma=\sigma(n,\al)$ be as in Lemma~\ref{lem:par-blowup-rot-est}. We have for every $0<r<r_0$
\begin{align*}
    &\int_{B_1}|u_{z_1,0}(x)-u_{z_2,0}(x)|G(x,-1)\,dx\\
    &\qquad\le \int_{B_1}|u_{z_1,0}(x)-u_{z_1,r}(x,-1)|G(x,-1)\,dx\\
    &\qquad\qquad+\int_{B_1}|u_{z_2,0}(x)-u_{z_2,r}(x,-1)|G(x,-1)\,dx\\
    &\qquad\qquad+\int_{B_1}|u_{z_1,r}(x,-1)-u_{z_2,r}(x,-1)|G(x,-1)\,dx\\
    &\qquad\le Cr^{2\si}+\frac{C(n)}{r^{3/2}}\int_{B_1}|u(x_1+rx,t_1-r^2)-u(x_2+rx,t_2-r^2)|\,dx.
\end{align*}
Since $(x_1+rx,t_1-r^2)$ and $(x_2+rx,t_2-r^2)$ are contained in $Q_{3/4}$ for every $x\in B_1$, we have by Theorem~\ref{thm:par-alm-min-grad-holder}
$$
|u(x_1+rx,t_1-r^2)-u(x_2+rx,t_2-r^2)|\le C\|z_1-z_2\|^{1/2}.
$$
By taking $r=\|z_1-z_2\|^{\frac1{4\si+3}}$ (which is possible if $r_1^{\frac1{4\si+3}}<r_0$), we obtain \begin{align*}
    \int_{B_1}|u_{z_1,0}(x)-u_{z_2,0}(x)|G(x,-1)\,dx&\le C\left(r^{2\si}+\frac{\|z_1-z_2\|^{1/2}}{r^{3/2}}\right)\\
    &=C(\ka_0,n,\al,u)\|z_1-z_2\|^\g,\quad \g=\frac{2\si}{4\si+3}.
\end{align*}
Now, the lemma follows by the boundedness of $G(x,-1)$ and the homogeneity of $u_{z_1,0}$ and $u_{z_2,0}$.
\end{proof}

The following lemma follows from Proposition~\ref{prop:par-gap-32-hom-classif} and Lemma~\ref{lem:par-blowup-est} by repeating the argument in \cite{GarPetSVG16}*{Lemma~7.5}.

\begin{lemma}\label{lem:par-blowup-hol}
Let $u\in\fs$ be an almost minimizer for the parabolic Signorini problem in $Q_1$, and $z_0\in\mathcal{R}(u)\cap Q'_{1/4}$. Then there exist $\rho>0$, depending on $z_0$, and $\g=\g(n,\al)>0$ such that $Q'_\rho(z_0)\cap \Ga(u)\subset \mathcal{R}(u)$ and if $u_{z_j,0}(x)=c_{z_j}\Re(x'\cdot e_{z_j}+i|x_n|)^{3/2}$ is the unique $3/2$-parabolically homogeneous blowup at $z_j\in Q'_\rho(z_0)\cap\Ga(u)$, $j=1,2$, then
\begin{align*}
    &|c_{z_1}-c_{z_2}|\le C_0|z_1-z_2|^\g,\\
    &|e_{z_1}-e_{z_2}|\le C_0|z_1-z_2|^\g
\end{align*}
with a constant $C_0$ depending on $z_0$.
\end{lemma}

We are now ready to prove the central result in this paper, the regularity of the regular set.

\begin{theorem}[Regularity of the regular set]\label{thm:par-Clg-reg-set}
Let $u\in \fs$ be an almost minimizer for the parabolic Signorini problem in $Q_1$. If $z_0=(x_0,t_0)\in \mathcal{R}(u)\cap Q'_{1/4}$, there exists $\rho>0$, depending on $z_0$, such that possibly after a rotation in $\R^{n-1}$, one has $Q'_\rho(z_0)\cap \Ga(u)\subset \mathcal{R}(u)$, and $$
Q'_\rho(z_0)\cap \Ga(u)=\{(x',t)\in Q'_\rho(z_0)\,:\, x_{n-1}=g(x'',t)\},
$$
for a function $g$ with $\D''g\in C^{\g,\g/2}$ for some $\g=\g(n,\al)\in(0,1)$.
\end{theorem}

\begin{proof}
Since the proof of this theorem follows the lines of \cite{JeoPet21}*{Theorem~9.7}, we shall provide only the outline of the proof.

Since $\mathcal{R}(u)$ is relatively open in $\Ga(u)$, we have $Q'_{2\rho}(z_0)\cap \Ga(u)\subset \mathcal{R}(u)$ for small $\rho>0$. We claim that for any $\e>0$, there exists $r_\e>0$ such that for any $\olz\in Q'_\rho(z_0)\cap \Ga(u)$ and $0<r<r_\e$, there holds 
\begin{align}\label{eq:u-rescal-conv}
\|u_{\olz,r}-u_{\olz,0}\|_{C^{1,0}(\overline{Q_1^\pm})}<\e.
\end{align}
Indeed, towards a contradiction, suppose there are sequences $r_j\to 0$ and $\olz_j\in Q'_\rho(z_0)\cap \Ga(u)$ such that for some $\e_0>0$
$$
\|u_{\olz_j,r_j}-u_{\olz_j,0}\|_{C^{1,0}(\overline{Q_1^\pm})}\ge \e_0.
$$ 
Up to a subsequence, we have $\olz_j\to \olz_0\in \overline{Q'_\rho(z_0)}\cap \Ga(u)$. We can argue as in the beginning of Section~\ref{sec:par-32-homog-blow} to deduce that over another subsequence
\begin{align}\label{eq:u-rescal-conv-2}
u_{\olz_j,r_j}\to w\quad\text{in }C^{1,0}(\overline{Q_1^\pm})
\end{align}
for some $w\in C^{1,0}(\overline{Q_1^\pm})$.
Moreover, we have by Lemma~\ref{lem:par-blowup-rot-est} that for any $s\in(-1,0)$
$$
u_{\olz_j,r_j}-u_{\olz_j,0}\to 0\quad\text{in }L^1(B_1\times(-1,s)),
$$
which implies by using Cantor's diagonal argument
$$
u_{\olz_j,r_j}-u_{\olz_j,0}\to 0\quad\text{a.e. in }Q_1.
$$
On the other hand, from Lemma~\ref{lem:par-blowup-hol}, we find
$$
u_{\olz_j,0}\to u_{\olz_0,0}\quad\text{in }C^{1,0}(\overline{Q_1^\pm}).
$$
The previous two convergences, combined with \eqref{eq:u-rescal-conv-2}, imply $w=u_{\olz_0,0}$ and contradict our assumption.

Next, for a given $\e>0$ and a unit vector $e\in \R^{n-1}$, define the cone 
$$
\mathcal{C}_\e(e)=\{x'\in\R^{n-1}\,:\, x'\cdot e>\e|x'|\}.
$$
By utilizing Lemma~\ref{lem:par-blowup-hol}, the estimate \eqref{eq:u-rescal-conv} and the complementarity condition (Lemma~\ref{lem:par-compl-cond}), we can follows Steps 2-3 in the proof of \cite{JeoPet21}*{Theorem~9.7} to obtain the following: for any $\e>0$, there is $r_\e>0$ such that for any $\olz=(\olx,\olt)\in Q'_\rho(z_0)\cap \Ga(u)$, we have 
\begin{align*}
&\olx+\left(\mathcal{C}_\e(e_{\olz})\cap B'_{r_\e}\right)\subset \{u(\cdot,0,\olt)>0\},\\
&\olx-\left(\mathcal{C}_\e(e_{\olz})\cap B'_{r_\e}\right)\subset \{u(\cdot,0,\olt)=0\}.
\end{align*}
Finally, by using these inclusions and Lemma~\ref{lem:par-blowup-hol}, we can repeat the arguments in Steps 4-5 in \cite{JeoPet21}*{Theorem~9.7} to conclude the theorem.
\end{proof}


\appendix

\section{Existence of weak solutions}\label{sec:appen-exist-sol}

In this section, we prove the existence and uniqueness of the weak solution to the parabolic Signorini problem in $S_1$, provided that the initial datum belongs to $W^2_\infty$.

\begin{theorem}\label{thm:exist-weak-sol}
If $\vp_0\in W^2_\infty(\R^n)$, then there exists a unique weak solution of 
\begin{align}\label{eq:weak-sol}\begin{cases}
& \pa_tv-\La v=0\quad\text{in } S^+_1\cup S^-_1,\\
& v\ge 0, \quad \pa_{\nu^+}v+\pa_{\nu^-}v\ge 0,\quad v(\pa_{\nu^+}v+\pa_{\nu^-}v)=0\quad\text{on } S_1',\\
&v(\cdot,-1)=\vp_0\quad\text{on }\R^n,
\end{cases}
\end{align}
where $\nu^{\pm}$ is the outer unit normal to $S_1^\pm$ on $S_1'$.
\end{theorem}

\begin{proof}
For the change of coordinates 
$$
\wv(y,\tau):=v\left(2e^{-\frac\tau2}y,-e^{-\tau}\right), \quad (y,\tau)\in \R^n\times [0,\infty),
$$
\eqref{eq:weak-sol} is equivalent to
\begin{align*}\begin{cases}
& \pa_\tau\tilde v+\frac y2\cdot\D\tilde v-\frac14\La \tilde v=0\quad\text{in } (\R^n_+\cup \R^n_-)\times(0,\infty),\\
& \tilde v\ge 0, \quad \pa_{\nu^+}\tilde v+\pa_{\nu^-}\tilde v \ge 0,\quad \tilde v(\pa_{\nu^+}\tilde v+\pa_{\nu^-}\tilde v)=0\quad\text{on }\R^{n-1}\times(0,\infty) ,\\
&\tilde v(\cdot,0)=\widetilde{\vp}_0\quad\text{on }\R^n,
\end{cases}
\end{align*}
where $\widetilde{\vp}_0(y)=\vp_0(2y)$. Note that $\tilde v$ is a weak solution of the above equation if and only if it satisfies for a.e. $\tau\in (0,\infty)$ the variational inequality \begin{align*}
    \int_{\R^n}\pa_\tau \tilde v(w-\tilde v)e^{-|y|^2}+\frac{y}2\cdot\D \tilde v(w-\tilde v)e^{-|y|^2}+\frac14\D \tilde v\cdot \D\left((w-\tilde v)e^{-|y|^2}\right)\ge 0,
\end{align*}
which is equivalent to  
\begin{align*}
    \int_{\R^n}\pa_\tau \tilde v(w-\tilde v)e^{-|y|^2}+\frac14\D \tilde v\cdot\D(w-\tilde v)e^{-|y|^2}\ge 0,
\end{align*}
for any $w\in L^2(0,\infty; W^{1,2}(\R^n, e^{-|y|^2}))$ with $w=\widetilde{\vp}_0$ on $\R^n\times\{0\}$, $w\ge 0$ on $\R^{n-1}\times(0,\infty)$ and $w-\tilde v\in L^2(0,\infty;W^{1,2}_0(\R^n,e^{-|y|^2}))$. In addition, for $a(v,v):=\frac14\int_{\R^n}\D v\cdot\D ve^{-|y|^2}\,dy$, the coercivity $$
a(v,v)+C\int_{\R^n}v^2e^{-|y|^2}\ge \al\int_{\R^n}(|\D v|^2+v^2)e^{-|y|^2}
$$
is satisfied. Therefore, the existence and the uniqueness of the weak solution $\tilde v$ follow from \cite{DuvLio76}*{Chapter~1, Theorem 5.1}.
\end{proof}

\section{Examples of almost minimizers}\label{sec:appen-ex}

In this section, we provide examples of solutions to certain equations that satisfy almost parabolic Signorini properties, both the unweighted and the weighted versions. These examples rely on the following technical lemma. For $\e\in(0,1)$, we write $Q_{r,\rho}^\e(z_0):=B_{r^\e}(x_0)\times(t_0-r^2,t_0-\rho^2]$.

\begin{lemma}\label{lem:par-ex-unweight-property}
For $\e=1/3$ and a point $z_0=(x_0,t_0)\in Q'_1$, suppose that a function $u\in W^{1,1}_2(Q_1)\cap L^2(-1,t_0;W^{1,2}(B_1,G_{z_0}))$ satisfies the following property: for any $Q_{r,\rho}^\e(z_0)\Subset Q_{1/2}$, and $v\in L^2(t_0-r^2,t_0-\rho^2;W^{1,2}(B_{r^\e}(x_0),G_{z_0}))$ with $v\ge0$ on $Q_{r,\rho}^\e(z_0)\cap Q_1'$ and $v-u\in L^2(t_0-r^2,r_0-\rho^2;W^{1,2}_0(B_{r^\e}(x_0),G_{z_0}))$
\begin{multline}\label{eq:par-ex-unweight-property}
    \int_{Q_{r,\rho}^\e(z_0)}\left((1-Cr^{\e\al})(t_0-t)|\D u|^2+((x_0-x)\cdot\D u+2(t_0-t)\pa_tu)(u-v)\right)G_{z_0}\\
    \le \int_{Q_{r,\rho}^\e(z_0)}\left((1+Cr^{\e\al})(t_0-t)|\D v|^2+Cr^{\e\al}\frac{|x_0-x|^2}{t_0-t}(u-v)^2\right)G_{z_0},
\end{multline}
where $C>0$ are constants, independent of $z_0$, $\rho$ and $r$. Let $\psi\in C_0^\infty(\R^n)$ be a cutoff function satisfying 
\begin{align*}
    &0\le\psi\le1,\quad \psi=1\quad\text{on }B_{1/2},\quad \supp\psi\subset B_1.
\end{align*}
Then there exists a constant $r_0>0$ such that $\wu:=u\psi$ satisfies the weighted almost parabolic Signorini property \eqref{eq:par-alm-min-def} at $z_0$ for $0\le \rho<r<r_0$, with a gauge function $\eta(r)=Cr^{\al/3}$.
\end{lemma}

\begin{remark}
    \label{rmk:norm-comp}
Since our main objective in this paper is the free boundary $\Ga(u)$, in Lemma~\ref{lem:par-ex-unweight-property}, we can make the assumption that $\|u\|_{W^{1,1}_2(Q_{1/2})}>0$. Otherwise, we have $u\equiv0$ in $Q_{1/2}$ and there is no free boundary on $Q_{1/2}'$. Moreover, the condition \eqref{eq:par-ex-unweight-property} only concerns $u$ within $Q^\e_{r,\rho}(z_0)$ and $Q^\e_{r,\rho}(z_0)\subset Q_{1/2}$, which allows us to freely modify the value of $u$ in $Q_1\setminus Q_{1/2}$. Therefore, we may assume that for some dimensional constant $C>0$
$$
\|u\|_{W^{1,1}_2(Q_1)}\le C\|u\|_{W^{1,1}_2(Q_{1/2})}.
$$
\end{remark}

\begin{proof}
\emph{Step 1.} Without loss of generality, we may assume $z_0=0$. \eqref{eq:par-ex-unweight-property} can be rewritten as \begin{align}
    \label{eq:par-ex-unweight-property-rewrite}
    (1-Cr^{\e\al})I+II\le (1+Cr^{\e\al})III+Cr^{\e\al}IV,
\end{align}
where \begin{align*}
    &I=\int_{Q^\e_{r,\rho}}(-t)|\D u|^2G,\quad II=\int_{Q^\e_{r,\rho}}(-x\cdot \D u-2t\pa_tu)(u-v)G,\\
    &III=\int_{Q^\e_{r,\rho}}(-t)|\D v|^2G,\quad IV=\int_{Q^\e_{r,\rho}}\frac{|x|^2}{-t}(u-v)^2G.\\
\end{align*}
For $0\le\rho<r<1$, let $w\in L^2(-r^2,-\rho^2;W^{1,2}(\R^n,G))$ with $w\ge 0$ on $S'_r\sm S'_\rho$ and $\tilde u-w\in L^2(-r^2,-\rho^2;W^{1,2}_0(\R^n,G))$. By approximation, we may assume that $w$ has a bounded support. We consider dilations of $\psi$
$$
\psi_r(x)=\psi_{r,\e}(x):=\psi\left(\frac x{r^\e}\right),
$$
and define
\begin{align}\label{eq:comp-def}
v(x,t):=u(x,t)+\psi_r(x)(w(x,t)-\wu(x,t)),\quad(x,t)\in B_{r^\e}\times(-r^2,-\rho^2).
\end{align}
Then $v-u\in L^2(-r^2,-\rho^2;W^{1,2}_0(B_{r^\e},G))$ and $v=u+w-\wu=w\ge 0$ on $Q^\e_{r,\rho}\cap Q'_1$. Thus $v$ is a valid competitor for $u$, and hence \eqref{eq:par-ex-unweight-property-rewrite} holds for such $v$. In the below we estimate and rewrite $I$, $II$, $III$ and $IV$ in terms of $\wu$ and $w$.

\medskip\noindent\emph{Step 2.} We first deal with $I$. We compute
\begin{align}\label{eq:par-ex-unweight-property-rewrite-I}\begin{split}
    &\int_{S_r\sm S_\rho}(-t)|\D\wu|^2G= \int_{B_1\times(-r^2,-\rho^2)}(-t)|\psi\D u+u\D\psi|^2G\\
    &=\int_{B_1\times(-r^2,-\rho^2)}(-t)|\D u|^2G\\
    &\qquad+\int_{(B_1\sm B_{1/2})\times(-r^2,-\rho^2)}(-t)\left( (\psi^2-1)|\D u|^2+2u\psi\D u\cdot\D\psi+u^2|\D\psi|^2\right)G\\
    &=\int_{Q^\e_{r,\rho}}(-t)|\D u|^2G+\int_{(B_1\sm B_{r^\e})\times(-r^2,-\rho^2)}(-t)|\D u|^2G\\
    &\qquad+\int_{(B_1\sm B_{1/2})\times(-r^2,-\rho^2)}(-t)\left( (\psi^2-1)|\D u|^2+2u\psi\D u\cdot\D\psi+u^2|\D\psi|^2\right)G.
\end{split}\end{align}
To estimate the last two terms, we claim that for $t\in(-r^2,-\rho^2)$ with $r<r_0=r_0(n)$ small, we have 
\begin{align}\label{eq:G-bound}
G(x,t)\le e^{\frac1{17t}}\,\,\,\text{for }|x|\ge 1/2\quad\text{and}\quad\text G(x,t)\le e^{\frac{1}{32t}r^{2\e}}\,\,\,\text{for }|x|\ge \frac12r^\e.
\end{align}
Indeed, if $|x|\ge 1/2$, then we easily have $$
G(x,t)\le \frac{e^{\frac{|x|^2}{4t}}}{(-t)^{n/2}}\le \frac{e^{\frac1{16t}}}{(-t)^{n/2}}\le e^{\frac1{17t}},
$$
which gives the first estimate. For the second one, we define $$
\zeta_r(s):=\frac{e^{-\frac{r^{2\e}}{32s}}}{s^{n/2}}, \quad 0<s\le r^2.
$$
Then $\zeta_r(r^2)=\frac{e^{-\frac1{32r^{2-2\e}}}}{r^n}<1$ and $\frac{d}{ds}\zeta_r(s)=\frac{ne^{-\frac{r^{2\e}}{32s}}}{2s^{\frac n2+2}}\left(\frac{r^{2\e}}{16n}-s\right)\ge 0$, $0<s\le r^2$, which gives $\zeta_r(s)<1$ for $0<s\le r^2$. Thus, if $|x|\ge \frac12r^\e$ and $t\in(-r^2,0)$, then 
$$
G(x,t)\le \frac{e^{\frac1{16t}r^{2\e}}}{(-t)^{n/2}}=\zeta_r(-t)e^{\frac1{32t}r^{2\e}}\le e^{\frac1{32t}r^{2\e}}.
$$
By using the claim \eqref{eq:G-bound}, we obtain
\begin{align*} &\left|\int_{(B_1\sm B_{1/2})\times(-r^2,-\rho^2)}(-t)\left( (\psi^2-1)|\D u|^2+2u\psi\D u\cdot\D\psi+u^2|\D\psi|^2\right)G\right|\\
&\qquad\le C(n)\int_{(B_1\sm B_{1/2})\times(-r^2,-\rho^2)}(|\D u|^2+u^2)G\\
&\qquad\le C(n)\|u\|^2_{W^{1,1}_2(Q_1)}e^{-\frac1{17r^2}},
\end{align*}
and \begin{align*}
    \int_{(B_1\sm B_{r^\e})\times(-r^2,-\rho^2)}(-t)|\D u|^2G\le C(n)\|u\|^2_{W^{1,1}_2(Q_1)}e^{-\frac1{32r^{2-2\e}}}.
\end{align*}
Combining these estimates with \eqref{eq:par-ex-unweight-property-rewrite-I} yields
\begin{align*}
   I\ge\int_{S_r\sm S_\rho}(-t)|\D\wu|^2G-C(n)\|u\|^2_{W^{1,1}_2(Q_1)}e^{-\frac1{32r^{2-2\e}}}.
\end{align*}

\medskip\noindent\emph{Step 3.} To estimate $II$, we observe that $u=\wu$ in $Q^\e_{r,\rho}$,  $\wu=0$ in $(\R^n\sm B_1)\times(-r^2,-\rho^2)$, and 
\begin{align}
    \label{eq:G-est}
    r^{-\e}G(x,t)\le r^{-\e}e^{\frac1{32t}r^{2\e}}\le r^{-\e}e^{-\frac1{16r^{2-2\e}}}\le e^{-\frac1{32r^{2-2\e}}}
\end{align}
for $(x,t)\in (B_1\sm B_{\frac12 r^\e})\times(-r^2,-\rho^2)$ with $r<r_0$ small. By using \eqref{eq:comp-def}, \eqref{eq:G-est} and Young's inequality, we deduce
\begin{align*}
    II&=\int_{Q^\e_{r,\rho}}\psi_r(-x\cdot\D\wu-2t\pa_t\wu)(\wu-w)G\\
    &=\int_{Q^\e_{r,\rho}}(-x\cdot\D\wu-2t\pa_t\wu)(\wu-w)G\\
    &\qquad+\int_{(B_{r^\e}\sm B_{\frac12r^\e})\times(-r^2,-\rho^2)}(\psi_r-1)(-x\cdot\D\wu-2t\pa_t\wu)(\wu-w)G\\
    &=\int_{S_r\sm S_\rho}(-x\cdot\D\wu-2t\pa_t\wu)(\wu-w)G\\
    &\qquad-\int_{(B_1\sm B_{r^\e})\times(-r^2,-\rho^2)}(-x\cdot\D\wu-2t\pa_t\wu)(\wu-w)G\\
    &\qquad+\int_{(B_{r^\e}\sm B_{\frac12r^\e})\times(-r^2,-\rho^2)}(\psi_r-1)(-x\cdot\D\wu-2t\pa_t\wu)(\wu-w)G\\
    &\ge \int_{S_r\sm S_\rho}(-x\cdot\D\wu-2t\pa_t\wu)(\wu-w)G-r^\e\int_{(B_1\sm B_{\frac12r^\e})\times(-r^2,-\rho^2)}(\wu-w)^2G\\
    &\qquad-r^{-\e}\int_{(B_1\sm B_{\frac12r^\e})\times(-r^2,-\rho^2)}(-x\cdot\D\wu-2t\pa_t\wu)^2G\\
    &\ge \int_{S_r\sm S_\rho}(-x\cdot\D\wu-2t\pa_t\wu)(\wu-w)G-r^\e\int_{S_r\sm S_\rho}(\wu-w)^2G\\
    &\qquad-\|u\|_{W^{1,1}_2(Q_1)}^2e^{-\frac1{33r^{2-2\e}}}.
\end{align*}

\medskip\noindent\emph{Step 4.} Before we estimate $III$, we prove \begin{multline}
    \label{eq:par-alm-min-cutoff-poin}
        \int_{(B_{r^\e}\sm B_{\frac12r^\e})\times(-r^2,-\rho^2)}|\D \psi_r|^2(w-\wu)^2G\\
         \le C(n)r^{4-6\e}\int_{S_r\sm S_\rho}(w-\wu)^2G+C(n)r^{2-4\e}\int_{S_r\sm S_\rho}(-t)(|\D w|^2+|\D\wu|^2)G.
\end{multline}
To this end, we apply the Log-Sobolev Ineuqality. \cite{EdqPet08}*{Lemma~1.2} can be rewritten as (by letting $g(y):=f\left(\frac y{\sqrt{-2t}}\right)$)
\begin{align*}
\log\left(\frac1{\mathcal{A}}\right)\int_{\R^n}g^2G(\cdot,t)\le -4t\int_{\R^n}|\D g|^2G(\cdot,t),\quad\text{where } \mathcal{A}:=\int_{\{|g|>0\}}G(\cdot,t),
\end{align*}
for every $t<0$ and $g\in W^{1,0}_2(\R^n,G(\cdot,t))$. We plug in $g=\pa_{x_i}\psi_r(w(\cdot,t)-\wu(\cdot,t))$, $1\le i\le n$, for each $t\in(-r^2,-\rho^2)$. Then, by using 
$$
\mathcal{A}\le\int_{B_{r^\e}\sm B_{\frac12r^\e}}G(x,t)\,dx\le e^{\frac1{32t}r^{2\e}}\int_{B_{r^\e}\sm B_{\frac12r^\e}}\,dx\le e^{\frac1{32t}r^{2\e}},
$$
where the second inequality holds due to \eqref{eq:G-bound}, we have 
\begin{align*}
    &\int_{\R^n\times\{t\}}|\pa_{x_i}\psi_r|^2(w-\wu)^2G\\
    &\qquad\le \frac{128t^2}{r^{2\e}}\int_{\R^n\times\{t\}}|\D(\pa_{x_i}\psi_r)(w-\wu)+\pa_{x_i}\psi_r\D(w-\wu)|^2G\\
    &\qquad\le \frac{C(n)t^2}{r^{2\e}}\left(\frac1{r^{4\e}}\int_{\R^n\times\{t\}}(w-\wu)^2G+\frac1{r^{2\e}}\int_{\R^n\times\{t\}}(|\D w|^2+|\D\wu|^2)  G\right)\\
    &\qquad=\frac{C(n)t^2}{r^{6\e}}\int_{\R^n\times\{t\}}(w-\wu)^2G+\frac{C(n)(-t)}{r^{4\e}}\int_{\R^n\times\{t\}}(-t)(|\D w|^2+|\D\wu|^2)G.
\end{align*}
By integrating in $t\in(-r^2,-\rho^2)$ and summing for $1\le i\le n$, we derive \eqref{eq:par-alm-min-cutoff-poin}.

Now, by using \eqref{eq:comp-def}, \eqref{eq:G-est}, \eqref{eq:par-alm-min-cutoff-poin} and Young's inequality, we have
 \begin{align*}
        III&=\int_{Q^\e_{r,\rho}}(-t)|\psi_r\D w+\D\psi_r(w-\wu)+\D u-\psi_r\D \wu|^2G\\
        &\le (1+r^\e)\int_{Q^\e_{r,\rho}}(-t)\psi_r^2|\D w|^2G\\
        &\qquad+\left(1+\frac1{r^\e}\right)\int_{\left(B_{r^\e}\sm B_{\frac12r^\e}\right)\times(-r^2,-\rho^2)}|\D\psi_r(w-\wu)+(1-\psi_r)\D u|^2G\\
       &\le (1+r^\e)\int_{S_r\sm S_\rho}(-t)|\D w|^2G\\
       &\qquad+\frac4{r^\e}\int_{\left(B_{r^\e}\sm B_{\frac12r^\e}\right)\times(-r^2,-\rho^2)}(|\D\psi_r(w-\wu)|^2+|\D u|^2)G\\
       &\le (1+r^\e+Cr^{2-5\e})\int_{S_r\sm S_\rho}(-t)|\D w|^2G+Cr^{2-5\e}\int_{S_r\sm S_\rho}(-t)|\D\wu|^2G\\
       &\qquad+Cr^{4-7\e}\int_{S_r\sm S_\rho}(w-\wu)^2G+C\|u\|_{W^{1,1}_2(Q_1)}^2e^{-\frac1{33r^{2-2\e}}}.
\end{align*}

\medskip \noindent\emph{Step 5.} It remains to consider $IV$. By using the equality $\D G=\frac{x}{2t}G$ and applying the integrations by parts and Young's inequality, we get
\begin{align*}
   &\int_{S_r\setminus S_\rho}\frac{|x|^2}{-t}(\wu-w)^2G\\
    &=\int_{S_r\setminus S_\rho}-2x(\wu-w)^2\D G=\int_{S_r\setminus S_\rho}\big[2n(\wu-w)^2+4x\cdot\D(\wu-w)(\wu-w)\big]G\\
    &\le 2n\int_{S_r\setminus S_\rho}(\wu-w)^2G+\frac12\int_{S_r\setminus S_\rho}\frac{|x|^2}{-t}(\wu-w)^2G+C\int_{S_r\setminus S_\rho}(-t)|\D(\wu-w)|^2G.
\end{align*}
This gives
\begin{align*}
   IV\le \int_{S_r\setminus S_\rho}\frac{|x|^2}{-t}(\wu-w)^2G\le C\int_{S_r\setminus S_\rho}(\wu-w)^2G+C\int_{S_r\setminus S_\rho}(-t)|\D(\wu-w)|^2G.
\end{align*}

\medskip\noindent\emph{Step 6.} By combining \eqref{eq:par-ex-unweight-property-rewrite} with the estimates for $I$-$IV$ and recalling $\e=1/3$, we obtain 
\begin{multline*}
    \int_{S_r\setminus S_\rho}\left[(1-Cr^{\al/3})(-t)|\D\tilde u|^2+(-x\cdot\D \tilde u-2t\pa_t\tilde u)(\tilde u-w)\right]G\\
    \le\int_{S_r\setminus S_\rho}\left[(1+Cr^{\al/3})(-t)|\D w|^2+Cr^{\al/3}(\tilde u-w)^2\right]G+C\|u\|_{W^{1,1}_2(Q_1)}e^{-\frac1{34r^{4/3}}}.
\end{multline*}
Finally, since $C\|u\|_{W^{1,1}_2(Q_1)}\le C\|u\|_{W^{1,1}_2(Q_{1/2})}\le C\|\tilde u\|_{W^{1,1}_2(Q_1)}$ by Remark~\ref{rmk:norm-comp}, we have for small $r_0>0$
$$
C\|u\|_{W^{1,1}_2(Q_1)}e^{-\frac1{34r^{4/3}}}\le \|\tilde u\|_{W^{1,1}_2(Q_1)}e^{-\frac1r},\quad 0<r<r_0.
$$
This completes the proof.
\end{proof}

Now we are ready to introduce some examples of almost minimizers, with the help of Lemma~\ref{lem:par-ex-unweight-property}.

\begin{example}\label{ex:par-alm-min-var} Given $0<\al<1$, let $A$ be a variable coefficient matrix satisfying $|A(x,t)-I|\le C(|x|^2+|t|)^{\al/2}$. Let $u\in W^{1,1}_2(Q_1)$ be a solution of the parabolic $A$-Signorini problem in $Q_1$
 \begin{align*}
    -\dv(A\D u)+ \pa_tu=0
    \quad\text{in }Q_1^\pm,\\
    \begin{multlined}u\geq 0,\quad \langle A\D u,\nu^+\rangle+\langle A\D u,\nu^-\rangle\geq 0,\\
     \qquad\qquad\qquad u(\langle A\D u,\nu^+\rangle+\langle A\D u,\nu^-\rangle)=0\quad\text{on }Q_1',
    \end{multlined}
  \end{align*}
where $\nu^\pm=\mp e_n$. We interpret this in the weak sense that $u$ satisfies for a.e. $t\in(-1,0)$ the variational inequality \begin{align}\label{eq:par-ex-var-ineq-stan}
\int_{B_1}\mean{A\D u,\D(u-w)}+\pa_tu(u-w)\le 0,
\end{align}
for any $w\in W^{1,2}(B_1)$ with $w=u$ on $\pa B_1$ and $w\ge 0$ on $B_1'$. Then 
\begin{enumerate}
    \item[(i)] $u$ satisfies the unweighted almost parabolic Signorini property at $0$ with a gauge function $\eta(r)=Cr^\al$.
    \item[(ii)] $\wu=u\psi$ satisfies the weighted almost parabolic Signorini property at $0$ with a gauge function $\eta(r)=Cr^{\al/3}$.
\end{enumerate}
\end{example}

\begin{proof}
We first treat (i). For any $0<r<1$, let $w\in W^{1,0}_2(Q_r)$ be such that $w=u$ on $\pa_pQ_r$ and $w\ge 0$ on $Q_r'$. By extending $w=u$ in $Q_1\sm Q_r$, we get from \eqref{eq:par-ex-var-ineq-stan} that $$
\int_{Q_r}\langle A\D u, \D(u-w)\rangle+\pa_tu(u-w)\le 0.
$$
Thus \begin{align*}
    &\int_{Q_r}|\D u|^2+\pa_tu(u-w)\\
    &\qquad=\int_{Q_r}\langle A\D u,\D(u-w)\rangle+\pa_tu(u-w)\\
    &\qquad\qquad+\int_{Q_r}\mean{\D u,\D w}+\mean{(A-I)\D u,\D w}+\mean{(I-A)\D u,\D u}\\
    &\qquad\le \frac12\int_{Q_r}(|\D u|^2+|\D w|^2)+\frac12\int_{Q_r}\left(r^{-\al}|(A-I)\D u|^2+r^\al|\D w|^2\right)\\
    &\qquad\qquad+\frac12\int_{Q_r}\left(r^{-\al}|(I-A)\D u|^2+r^\al|\D u|^2\right)\\
    &\qquad\le \frac{1+Cr^\al}2\int_{Q_r}(|\D u|^2+|\D w|^2).
\end{align*}
This gives the unweighted almost parabolic Signorini property of $u$ at $0$.

\medskip To prove the weighted property (ii), we observe that $u$ also satisfies for a.e. $t\in (-1,0)$ the following variational inequality \begin{align}
    \label{eq:par-ex-var-ineq-weight}
    \int_{B_1}[(-2t)\mean{A\D u,\D(u-v)}-\mean{x,A\D u}(u-v)+(-2t)\pa_tu(u-v)]G(\cdot,t)\le 0,
\end{align}
for any competitor $v\in W^{1,0}_2(B_1,G)$ with $v=u$ on $\pa B_1$ and $v\ge 0$ on $B'_1$. In fact, this follows by inserting $w=u+(v-u)e^{\frac{|x|^2}{4t}}$  in \eqref{eq:par-ex-var-ineq-stan} and multiplying $\frac{-2t}{(-4\pi t)^{n/2}}$ in both sides. To prove \eqref{eq:par-ex-unweight-property} for $z_0=0$, which readily implies (ii) by Lemma~\ref{lem:par-ex-unweight-property}, we fix $\e=1/3$. Then, for any $0\le\rho<r<1$ and $v\in W^{1,0}_2(Q^\e_{r,\rho},G)$ such that $v-u\in L^2(-r^2,-\rho^2;W^{1,2}_0(B_{r^\e}))$ and $v\ge 0$ on $Q^\e_{r,\rho}\cap Q'_1$, we extend $v=u$ in $(B_1\sm B_{r^\e})\times(-r^2,-\rho^2)$ and use \eqref{eq:par-ex-var-ineq-weight} to obtain
\begin{align*}
    \int_{Q^\e_{r,\rho}}(-2t)\mean{A\D u,\D(u-v)}G-\mean{x,A\D u}(u-v)G+(-2t)\pa_tu(u-v)G\le 0.
\end{align*}
Using $2\D u\cdot \D(u-v)\ge|\D u|^2-|\D v|^2$, $|\D u\cdot\D(u-v)|\le 3/2|\D u|^2+|\D v|^2$, and $|A-I|\le Cr^{\e\al}$ in $Q^\e_{r,\rho}$, we get 
\begin{align*}
    &\int_{Q^\e_{r,\rho}}(-2t)\mean{A\D u,\D(u-v)}G\\
    &\qquad=\int_{Q^\e_{r,\rho}}(-2t)\D u\cdot\D(u-v)G+\int_{Q^\e_{r,\rho}}(-2t)\mean{(A-I)\D u,\D(u-v)}G\\
    &\qquad\ge \left(1-Cr^{\e\al}\right)\int_{Q^\e_{r,\rho}}(-t)|\D u|^2G-\left(1+Cr^{\e\al}\right)\int_{Q^\e_{r,\rho}}(-t)|\D v|^2G.
\end{align*}
Combining the above two estimates yields
\begin{multline*}
    \left(1-Cr^{\e\al}\right)\int_{Q^\e_{r,\rho}}(-t)|\D u|^2G+\int_{Q^\e_{r,\rho}}(-x\cdot\D u-2t\pa_tu)(u-v)G\\
    \le\left(1+Cr^{\e\al}\right)\int_{Q^\e_{r,\rho}}(-t)|\D v|^2G+\int_{Q^\e_{r,\rho}}\mean{x,(A-I)\D u}(u-v)G.
\end{multline*}
Finally, by estimating the last term with Young's inequality
$$
\int_{Q^\e_{r,\rho}}\mean{x,(A-I)\D u}(u-v)G\le Cr^{\e\al}\int_{Q^\e_{r,\rho}}(-t)|\D u|^2G+\frac{|x|^2}{-t}(u-v)^2G,
$$
we conclude \eqref{eq:par-ex-unweight-property} for $z_0=0$.
\end{proof}

\begin{example}\label{ex:par-alm-min-drift}
Let $u$ be a solution of the parabolic Signorini problem for the Laplacian with drift with the velocity field $b\in L^\infty(-1,0\,;\,L^p(B_1))$, $p>n$: \begin{align*}
    -\La u+b(x,t)\cdot\D u+\pa_tu=0\quad&\text{in }Q_1^\pm\\
    -\pa_{x_n}u\ge 0,\,\,\,u\ge 0\,\,\,u\pa_{x_n}u=0\quad&\text{on }Q'_1,
\end{align*}
even in $x_n$-variable. We understand this in the weak sense that $u$ satisfies the variational inequality: for any $-1<t<0$, $$
\int_{B_1\times\{t\}}\D u\cdot\D(v-u)+b(x,t)\cdot\D u(v-u)+\pa_tu(v-u)\ge 0,
$$
for any competitor $v\in W^{1,2}(B_1)$ such that $v\ge 0$ on $B'_1$ and $v=u$ on $\pa B_1$. Then \begin{enumerate}
    \item[(i)] $u$ is an unweighted almost minimizer for the parabolic Signorini problem in $Q_1$ with a gauge function $\eta(r)=Cr^{1-n/p}$.
    \item[(ii)] $\tilde u=u\psi$ is a weighted almost minimizer for the parabolic Signorini problem in $Q'_1$ with a gauge function $\eta(r)=Cr^{\frac13(1-n/p)}$.
\end{enumerate}
\end{example}

\begin{proof}
Since (i) is proved in \cite{JeoPet21}*{Example~A.1} for more general case with variable coefficients, it is sufficient to prove (ii). For this purpose, as in Example~\ref{ex:par-alm-min-var}, we prove \eqref{eq:par-ex-unweight-property} for every $z_0\in Q'_1$. Indeed, without loss of generality we may assume that $z_0=0$. By the similar argument as in Example~\ref{ex:par-alm-min-var}, $u$ also satisfies for a.e. $t\in(-1,0)$ the variational inequality 
\begin{multline*}
\int_{B_1\times\{-t\}}\big[(-2t)\D u\cdot\D(u-v)
+(-x\cdot\D u-2t\pa_tu)(u-v)\\+(-2t)b\cdot\D u(u-v)\big]G\le 0
\end{multline*}
for any $v\in W^{1,2}(B_1,G(\cdot,t))$ with $v=u$ on $\pa B_1$ and $v\ge 0$ on $B'_1$. For $\e=1/3$ and $0\le\rho<r<1$, let $v\in W^{1,0}_2(Q^\e_{r,\rho},G)$ be such that $v-u\in L^2(-r^2,-\rho^2;W^{1,2}_0(B_{r^\e},G))$ and $v\ge 0$ on $Q^\e_{r,\rho}\cap Q'_1$. Extending $v $ to $B_1\times(-r^2,-\rho^2)$ by $v=u$ on $(B_1\sm B_{r^\e})\times(-r^2,-\rho^2)$ and using the above variational inequality, we get 
\begin{align*}
    &\int_{Q^\e_{r,\rho}}\left((-t)|\D u|^2+(-x\cdot\D u-2t\pa_tu)(u-v)\right)G\\
    &\quad\le \int_{Q^\e_{r,\rho}}\left((-t)|\D u|^2+(-2t)\D u\cdot\D(v-u)+(-2t)b\cdot\D u(v-u)\right)G\\
    &\quad=\int_{Q^\e_{r,\rho}}\left(-(-t)|\D u|^2+(-2t)\D u\cdot\D v\right)G+\int_{-r^2}^{-\rho^2}(-2t)\int_{B_{r^\e}}b\cdot\D u(v-u)G\,dxdt\\
    &\quad\le \int_{Q^\e_{r,\rho}}(-t)|\D v|^2G+\int_{-r^2}^{-\rho^2}(-2t)M\|\D uG^{1/2}\|_{L^2(B_{r^\e})}\|(v-u)G^{1/2}\|_{L^{p^*}(B_{r^\e})}\,dt,
\end{align*}
where $M:=\sup\{\|b(\cdot,t)\|_{L^p(B_1)}:-1<t<0\}$ and $p^*=\frac{2p}{p-2}$.
For $\g=1-n/p$, we have by Sobolev's inequality,
\begin{align*}
    \|(v-u)G^{1/2}\|_{L^{p^*}(B_{r^\e})}&\le C_{n,p}r^{\e\g}\|\D((v-u)G^{1/2})\|_{L^2(B_{r^\e})}\\
    &\le Cr^{\e\g}\left(\|\D(v-u)G^{1/2}\|_{L^2(B_{r^\e})}+\|(v-u)\frac xtG^{1/2}\|_{L^2(B_{r^\e})}\right).
\end{align*}
Thus
\begin{align*}
    &\int_{-r^2}^{-\rho^2}(-2t)M\|\D uG^{1/2}\|_{L^2(B_{r^\e})}\|(v-u)G^{1/2}\|_{L^{p^*}(B_{r^\e})}\,dt\\
    &\qquad\begin{multlined}
        \le Cr^{\e\g}\int_{-r^2}^{-\rho^2}(-2t)\|\D uG^{1/2}\|_{L^2(B_{r^\e})}\Big(\|\D(v-u)G^{1/2}\|_{L^2(B_{r^\e})}\\
        +\|(v-u)\frac xtG^{1/2}\|_{L^2(B_{r^\e})}\Big)dt\end{multlined}\\
    &\qquad\le Cr^{\e\g}\int_{Q^\e_{r,\rho}}(-2t)\left(|\D u|^2+|\D v|^2\right)G+Cr^{\e\g}\int_{Q^\e_{r,\rho}}\frac{|x|^2}{(-t)}(v-u)^2G,
\end{align*}
where constants $C>0$ depend only on $n,p$ and $M$.
Therefore, 
\begin{multline*}
    \int_{Q^\e_{r,\rho}}\left(1-Cr^{\e\g}\right)(-t)|\D u|^2G+(-x\cdot\D u-2t\pa_tu)(u-v)G\\
    \le \int_{Q^\e_{r,\rho}}\left(1+Cr^{\e\g}\right)(-t)|\D v|^2G+Cr^{\e\g}\frac{|x|^2}{(-t)}(v-u)^2G.
\end{multline*}
This completes the proof.
\end{proof}


{\bf Disclosure statement.} The authors report there are no competing interests to declare.



\begin{bibdiv}
\begin{biblist}


\bib{AryShi24}{article}{
   author={Arya, Vedansh},
   author={Shi, Wenhui},
   title={Optimal regularity for the variable coefficients parabolic Signorini problem},
   pages={67 pp},
   date={2024},
   status={arXiv:2401.13305 preprint},
 }



\bib{AthCafSal08}{article}{
   author={Athanasopoulos, I.},
   author={Caffarelli, L. A.},
   author={Salsa, S.},
   title={The structure of the free boundary for lower dimensional obstacle
   problems},
   journal={Amer. J. Math.},
   volume={130},
   date={2008},
   number={2},
   pages={485--498},
   issn={0002-9327},
   review={\MR{2405165}},
   doi={10.1353/ajm.2008.0016},
}


\bib{CafSalSil08}{article}{
   author={Caffarelli, Luis A.},
   author={Salsa, Sandro},
   author={Silvestre, Luis},
   title={Regularity estimates for the solution and the free boundary of the
   obstacle problem for the fractional Laplacian},
   journal={Invent. Math.},
   volume={171},
   date={2008},
   number={2},
   pages={425--461},
   issn={0020-9910},
   review={\MR{2367025}},
   doi={10.1007/s00222-007-0086-6},
 }



\bib{DanGarPetTo17}{article}{
   author={Danielli, Donatella},
   author={Garofalo, Nicola},
   author={Petrosyan, Arshak},
   author={To, Tung},
   title={Optimal regularity and the free boundary in the parabolic
   Signorini problem},
   journal={Mem. Amer. Math. Soc.},
   volume={249},
   date={2017},
   number={1181},
   pages={v + 103},
   issn={0065-9266},
   isbn={978-1-4704-2547-0},
   isbn={978-1-4704-4129-6},
   review={\MR{3709717}},
   doi={10.1090/memo/1181},
 }
 
  
\bib{DuvLio76}{book}{
   author={Duvaut, Georges},
   author={Lions, Jacques-Louis},
   title={Inequalities in mechanics and physics},
   note = {Translated from the French by C. W. John,Grundlehren der Mathematischen Wissenschaften, 219},
   publisher={Springer-Verlag, Berlin-New York},
   date={1976},
   pages={xvi+397},
   isbn={3-540-07327-2},
   review={\MR{0521262}},
}





\bib{EdqPet08}{article}{
   author={Edquist, Anders},
   author={Petrosyan, Arshak},
   title={A parabolic almost monotonicity formula},
   journal={Math. Ann.},
   volume={341},
   date={2008},
   number={2},
   pages={429--454},
   issn={0025-5831},
   review={\MR{2385663}},
}


\bib{GarPet09}{article}{
   author={Garofalo, Nicola},
   author={Petrosyan, Arshak},
   title={Some new monotonicity formulas and the singular set in the lower
   dimensional obstacle problem},
   journal={Invent. Math.},
   volume={177},
   date={2009},
   number={2},
   pages={415--461},
   issn={0020-9910},
   review={\MR{2511747}},
   doi={10.1007/s00222-009-0188-4},
}


\bib{GarPetSVG16}{article}{
   author={Garofalo, Nicola},
   author={Petrosyan, Arshak},
   author={{Smit Vega Garcia}, Mariana},
   title={An epiperimetric inequality approach to the regularity of the free
   boundary in the Signorini problem with variable coefficients},
   language={English, with English and French summaries},
   journal={J. Math. Pures Appl. (9)},
   volume={105},
   date={2016},
   number={6},
   pages={745--787},
   issn={0021-7824},
   review={\MR{3491531}},
   doi={10.1016/j.matpur.2015.11.013},
}

 
 
\bib{JeoPet21}{article}{
   author={Jeon, Seongmin},
   author={Petrosyan, Arshak},
   title={Almost minimizers for the thin obstacle problem},
   journal={Calc. Var. Partial Differential Equations},
   volume={60},
   date={2021},
   pages={Paper No. 124, 59},
   issn={0944-2669},
   review={\MR{4277328}},
   doi={10.1007/s00526-021-01986-8},
 }


\bib{JeoPet23}{article}{
   author={Jeon, Seongmin},
   author={Petrosyan, Arshak},
   title={Regularity of almost minimizers for the parabolic thin obstacle problem},
   journal={Nonlinear Anal.},
   volume={237},
   date={2023},
   pages={Paper No. 113386, 29},
   issn={0362-546X},
   review={\MR{4645649}},
   doi={10.1016/j.na.2023.113386},
 }


\bib{JeoPetSVG24}{article}{
   author={Jeon, Seongmin},
   author={Petrosyan, Arshak},
   author={{Smit Vega Garcia}, Mariana},
   title={Almost minimizers for the thin obstacle problem with variable coefficients},
   journal={Interfaces Free Bound.},
   volume={},
   date={2024},
   pages={},
   issn={ },
   review={ },
   doi={10.4171/IFB/507},
 }


\bib{PetShaUra12}{book}{
   author={Petrosyan, Arshak},
   author={Shahgholian, Henrik},
   author={Uraltseva, Nina},
   title={Regularity of free boundaries in obstacle-type problems},
   series={Graduate Studies in Mathematics},
   volume={136},
   publisher={American Mathematical Society, Providence, RI},
   date={2012},
   pages={x+221},
   isbn={978-0-8218-8794-3},
   review={\MR{2962060}},
   doi={10.1090/gsm/136},
 }
 
 
\bib{Poo96}{article}{
 author={Poon, Chi-Cheung},
 title={Unique continuation for parabolic equations},
 journal={Comm. Partial Differential Equations}, 
 volume={21},
 date={1996},
 pages={521--539},
 issn={0360-5302},
 review={\MR{1387458}},
 doi={10.1080/03605309608821195},
 }

  
 \bib{Shi20}{article}{
 author={Shi, Wenhui},
 title={An epiperimetric inequality approach to the parabolic Signorini problem},
 journal={Disc. Cont. Dyn. Syst.}, 
 volume={40},
 date={2020},
 pages={1813--1846},
 issn={1078-0947},
 review={\MR{4063946}},
 doi={10.3934/dcds.2020095},
 }

\bib{Wei99}{article}{
   author={Weiss, Georg S.},
   title={A homogeneity improvement approach to the obstacle problem},
   journal={Invent. Math.},
   volume={138},
   date={1999},
   number={1},
   pages={23--50},
   issn={0020-9910},
   review={\MR{1714335}},
   doi={10.1007/s002220050340},
 }


\end{biblist}
\end{bibdiv}
\end{document}